\documentclass[11pt]{amsart}

\usepackage{amsfonts}
\usepackage{amssymb}
\usepackage{amsmath}
\usepackage{amsthm}
\usepackage{stmaryrd} % contains some logic/CS symbols
\usepackage{wasysym} % funny symbols
\usepackage[colorlinks,allcolors=magenta]{hyperref}
\usepackage{tikz}

\usepackage[all]{xy}
\usepackage{enumitem}

\setenumerate{topsep=0pt,labelwidth=0pt,align=left,labelindent=0pt,labelsep=0in,labelwidth=0.275in,leftmargin=0.275in,label=\rm(\alph*),parsep=0pt,itemsep=3pt}

\renewcommand{\phi}{\varphi}

\renewcommand{\bar}{\overline}
\newcommand{\restr}{\upharpoonright}
\newcommand{\forces}{\mathrel{\Vdash}}

\newcommand{\Q}{\mathbb{Q}}
\newcommand{\R}{\mathbb{R}}
\newcommand{\C}{\mathbb{C}}

\newcommand{\B}{\mathbb{B}}
\newcommand{\A}{\mathbb{A}}

\renewcommand{\P}{\mathbb{P}}

\newcommand{\LM}{\mathcal{M}}

\newcommand{\LF}{\mathcal{F}}
\newcommand{\LP}{\mathcal{P}}
\newcommand{\LC}{\mathcal{C}}

\newcommand{\LA}{\mathcal{A}}
\newcommand{\LB}{\mathcal{B}}

\newcommand{\LD}{\mathcal{D}}
\newcommand{\LG}{\mathcal{G}}

\newcommand{\p}{\mathfrak{p}}

\renewcommand{\b}{\mathfrak{b}}
\renewcommand{\d}{\mathfrak{d}}
\renewcommand{\a}{\mathfrak{a}}
\renewcommand{\c}{\mathfrak{c}}

\DeclareMathOperator{\dom}{dom}

\DeclareMathOperator{\ran}{ran}

\newtheorem{thm}{Theorem}[section]
\newtheorem{prop}[thm]{Proposition}
\newtheorem{lemma}[thm]{Lemma}
\newtheorem{cor}[thm]{Corollary}

\newtheorem*{ques}{Question}
\newtheorem*{conj}{Conjecture}

\theoremstyle{definition}
\newtheorem{defn}[thm]{Definition}

\theoremstyle{remark}

\newtheorem*{claim}{Claim}

\newcommand{\LH}{\mathcal{H}}

\newcommand{\supp}{\mathrm{supp}}
\newcommand{\cf}{\mathrm{cf}}

\newcommand{\concat}{^\smallfrown}

\newcommand{\ZFC}{\mathsf{ZFC}}
\newcommand{\CH}{\mathsf{CH}}
\newcommand{\MA}{\mathsf{MA}}

\newcommand{\bb}{\mathrm{bb}}

\renewcommand{\P}{\mathbb{P}}

\newcommand{\V}{\mathbf{V}}
\renewcommand{\L}{\mathbf{L}}

\newcommand{\FIN}{\mathrm{FIN}}

\newcommand{\cc}{\mathfrak{c}}

\newcommand{\FU}{\mathrm{FU}}

\title{Madness in vector spaces}
\date{August 20, 2019}
\author{Iian B. Smythe}
\address{Department of Mathematics, Rutgers, The State University of New Jersey, 110 Frelinghuysen Road, Piscataway, NJ, 08854}
\email{i.smythe@rutgers.edu}

\thanks{The author would like to thank the anonymous referee for many helpful comments and corrections.}

\begin{document}

\begin{abstract}
	We consider maximal almost disjoint families of block subspaces of countable vector spaces, focusing on questions of their size and definability. We prove that the minimum infinite cardinality of such a family cannot be decided in $\ZFC$ and that the ``spectrum'' of cardinalities of mad families of subspaces can be made arbitrarily large, in analogy to results for mad families on $\omega$. We apply the author's local Ramsey theory for vector spaces \cite{MR3864398} to give partial results concerning their definability.
\end{abstract}

\maketitle

\section{Introduction}\label{sec:Intro}

Recall that two infinite subsets $x$ and $y$ of the natural numbers $\omega$ are \emph{almost disjoint} if $x\cap y$ is finite. A collection $\LA\subseteq[\omega]^\omega$, where $[\omega]^\omega$ is the set of infinite subsets of $\omega$, is an \emph{almost disjoint family} if its elements are pairwise almost disjoint, and is a \emph{maximal almost disjoint family}, or \emph{mad family}, if it is not properly contained in another such family. While any finite (almost) partition of $\omega$ forms a mad family, our focus here is confined to infinite mad families.

It is well-known that every almost disjoint family is contained in a mad family and every infinite mad family is uncountable. The former is an application of Zorn's Lemma, while the later a straightforward diagonalization.

A large almost disjoint family can be obtained as follows: Identifying $\omega$ with $2^{<\omega}$, consider
\begin{align}\label{ex:perfect_ad}
	\LA=\{\{x\restr n:n\in\omega\}:x\in 2^\omega\}.
\end{align}
It is easy to see that $\LA$ is almost disjoint and of size $\cc$, thus can be extended to a mad family of size $\cc$. Note that $\LA$ is (topologically) closed as it is a homeomorphic image of $2^\omega$. Here, we identify $[\omega]^\omega$ as a subspace of $2^\omega$ via characteristic functions, from which it inherits a Polish topology.

Two fundamental questions about infinite mad families one might ask are:
\begin{enumerate}[label=\textup{\arabic*.}]
	\item How big (or small) can they be?
	\item How definable can they be?	
\end{enumerate}

One way of addressing question 1 is to determine the value of the cardinal invariant
\[
	\a = \min\{|\LA|:\LA\text{ is an infinite mad family}\}.
\]
This could mean which $\aleph_\alpha$ is such that $\a=\aleph_\alpha$, or how $\a$ relates to other well-studied cardinal invariants (see \cite{MR2768685} or \cite{MR776622}) between $\aleph_1$ and $\cc$. By our comments above, $\aleph_1\leq\a\leq\cc$, and a modification of that diagonalization argument shows that $\b\leq\a$, where $\b$ is the minimum size of an unbounded family of functions $\omega\to\omega$ (see [ibid.]). However, the value of $\a$ cannot be decided in $\ZFC$: both the Continuum Hypothesis $\CH$ and Martin's Axiom $\MA$ (see \cite{MR597342} or \cite{MR0270904}) imply that $\a=\cc$, and thus, consistently $\aleph_1<\a=\cc$, while Kunen \cite{MR597342} showed that in the model obtained by adding $\aleph_2$-many Cohen reals to a model of $\CH$, $\aleph_1=\a<\cc=\aleph_2$. In \cite{MR1900391}, Hru\v{s}\'ak showed\footnote{Given the comments in \cite{MR1900391}, this result was likely known earlier.} that the latter also holds in the models obtained by adding $\aleph_2$-many Sacks reals iteratively or ``side-by-side'' to a model of $\CH$. 

A more sophisticated version of question 1 might ask for the ``spectrum'' of cardinalities between $\aleph_1$ and $\cc$ that mad families can posses. This was first addressed by Hechler \cite{MR0307913}, who produced a method for obtaining arbitrarily large continuum and, simultaneously, mad families of all cardinalities $\kappa$ for $\aleph_1\leq\kappa\leq\cc$. While beyond the scope of our investigations here, these questions have been the focus of much deep work in recent decades, notably Brendle's \cite{MR1975392}, which establishes the consistency of $\a=\aleph_\omega$, Shelah's \cite{MR2096454}, which establishes the consistency of $\d<\a$, and Shelah and Spinas' \cite{MR3395354}, which gives a nearly-sharp characterization of possible mad spectra.

Question 2 above seeks to understand to what extent the nonconstructive methods used to obtain mad families are necessary. A result of Mathias \cite{MR0491197} says that an infinite mad family can never be analytic (i.e., a continuous image of a Borel set). Under large cardinal hypotheses, this can be pushed further to show that there are no definable mad families at all, in the sense that there are none in $\L(\R)$ (see \cite{MR1644345}, \cite{MR0491197}, \cite{MR3787540}, and for a consistency result without large cardinals, \cite{Horowitz_Shelah_no_mad}). Mathias' result is also sharp; Miller \cite{MR983001} proved that there is a coanalytic (i.e., the complement of an analytic set) mad family assuming $\V=\L$, work later refined by T\"ornquist \cite{MR3156517}.
 
This article is concerned with an analogue of mad families arising in vector spaces. Throughout, $E$ will be a countably infinite-dimensional vector space over a countable (possibly finite) field $F$.

\begin{defn}
	We say that two infinite-dimensional subspaces $X$ and $Y$ of $E$ are \emph{almost disjoint} if $X\cap Y$ is finite-dimensional.	
\end{defn}

Due to their more tractable nature, we will focus on \emph{block subspaces}, that is, those having a basis in ``block position'' with respect to a fixed basis for $E$ (see \S\ref{sec:card_ZFC} for the definition). Every infinite-dimensional subspace contains a block subspace, so this is a relatively mild restriction.

\begin{defn}
	A collection $\LA$ of infinite-dimensional (block) subspaces of $E$ is an \emph{almost disjoint family of (block) subspaces} if its elements are pairwise almost disjoint and is a \emph{maximal almost disjoint family of (block) subspaces}, or \emph{mad family of (block) subspaces}, if it is not properly contained in another such family.
\end{defn}

Note that, by our prior comment, being maximal with respect to arbitrary subspaces is equivalent to being maximal with respect to block subspaces.

While the topic of almost disjoint families of subspaces seems very natural, it appears to have been little studied except for a paper by Kolman \cite{MR1756262}, wherein they are called ``almost disjoint packings'',\footnote{Several proofs in \cite{MR1756262} appear to use a stronger property than almost disjointness, namely that whenever $X_0,\ldots,X_n\in\LA$ are distinct, then $X_i\cap(\sum_{j\neq i}X_j)$ is finite-dimensional. It easy to construct almost disjoint families of subspaces for which this fails, e.g., $X_0=\langle(e_{2n})_{n\in\omega}\rangle$, $X_1=\langle(e_{2n+1})_{n\in\omega}\rangle$, and $X_2=\langle(e_{2n}+e_{2n+1})_{n\in\omega}\rangle$, where $(e_n)$ is a basis for $E$. This can be extended to an infinite almost disjoint family of subspaces by our Proposition \ref{prop:mad_unctbl}. As such, we reprove several of the results appearing in \cite{MR1756262}.} and indirectly in the recent work of Brendle and Garc\'{i}a \'{A}vila \cite{MR3685044} discussed below.

In light of the above questions for mad families on $\omega$, we ask the analogous questions for infinite mad families of subspaces:

\begin{enumerate}[label=\textup{\arabic*.}]
	\item How big (or small) can they be? In particular, what is
	\[
		\a_{\mathrm{vec},F}=\min\{|\LA|:\LA\text{ is an infinite mad family of block subspaces}\}?%\footnote{Technically, $\a_{\mathrm{vec},F}$ might depend on the field $F$, though we will suppress this throughout.}
	\]
	\item How definable can they be?	
\end{enumerate}

Two related notions have been studied for separable Hilbert spaces, that of ``almost orthogonal'' and ``almost disjoint'' families of closed infinite-dimensional subspaces, where ``almost'' is measured by considering the corresponding projection operators modulo the compact operators. Results concerning question 1 in these settings were obtained in papers of Wofsey \cite{MR2358514} and Bice \cite{MR2847547}, respectively. While not directly related\footnote{Almost orthogonal families of closed subspaces of Hilbert space appear more closely related to almost disjoint families on $\omega$ than does our setting. For instance, countable almost orthogonal families arise as images of countable almost disjoint families on $\omega$ via the ``diagonal map'' (cf.~Lemma 5.34 in \cite{FarahAST}), and, consistently, some mad families on $\omega$ remain maximal when passed through this map \cite{MR2358514}. Less is understood about the notion of almost disjointness for closed subspaces, e.g., it appears to be open whether the corresponding cardinal invariant is $\aleph_1$ in $\ZFC$.} to our setting, these papers provide both motivation for, and ideas used in, the results in \S\ref{sec:card_cons} below.

When $F$ is the finite field of order $2$, vectors may be identified with elements of $\FIN$, the set of nonempty finite subsets of $\omega$, via their supports. Sums of vectors in block position correspond to unions of the corresponding supports. This is the setting of Hindman's Theorem \cite{MR0349574} on disjoint unions of finite subsets of $\omega$.  During the preparation of this article, an independent work of Brendle and Garc\'{i}a \'{A}vila \cite{MR3685044} appeared on maximal almost disjoint families of combinatorial subspaces of $\FIN$. Among other results, they show that $\mathrm{non}(\LM)\leq\a_{\FIN}$, where $\mathrm{non}(\LM)$ is the minimum size of a nonmeager subset of $\R$ and $\a_{\FIN}$ is the minimum size of an infinite mad family in $\FIN$, or in our language, a mad family of block subspaces when $|F|=2$. Together with known results, this shows the consistency of $\a<\a_{\FIN}$. %We have adaptedTheir work has allowed us to establish the corresponding results for \emph{block} subspaces, Corollaries \ref{cor:nonM<=a_vec} and \ref{cor:a<a_vec}, for general $F$. This addresses a question from an earlier version of the present article, appearing in \cite{Smythe_thesis}.

%However, the particular choice of $F$ will not affect our main results.

This article is organized as follows: In \S\ref{sec:card_ZFC}, we consider issues of cardinality and address question 1 using only $\ZFC$ techniques, showing that mad families of block subspaces of cardinality $\geq 2$ are always uncountable (Proposition \ref{prop:mad_unctbl}) and that $\b\leq \a_{\mathrm{vec},F}$ (Proposition \ref{prop:bleqa}). We then adapt the aforementioned work of Brendle and Garc\'{i}a \'{A}vila to show, moreover, that $\mathrm{non}(\LM)\leq\a_{\mathrm{vec},F}$, for general $F$ (Corollary \ref{cor:nonMleqa_vec}). In \S\ref{sec:card_cons}, we use forcing to establish consistency results regarding $\a_{\mathrm{vec},F}$ in analogy to those mentioned above for $\a$, showing that is equal to $\aleph_1$ in the Cohen (Theorem \ref{thm:cohen}) and Sacks (Theorem \ref{thm:sacks}) models, and that the spectrum of cardinalities of mad families of block subspaces can be made arbitrarily large (Theorem \ref{thm:hechler}). In \S\ref{sec:def}, we consider issues of definability. We use the Ramsey-theoretic results from the author's \cite{MR3864398} to give a partial solution for ``full'' mad families of subspaces (Theorem \ref{thm:no_full_mad}). The existence of such families is established under certain set-theoretic hypotheses (Theorem \ref{thm:full_mad}). \S\ref{sec:def} can be read independently from the other sections. We conclude in \S\ref{sec:fin} with further remarks, conjectures, and open questions.

\section{Cardinality: $\ZFC$ results}\label{sec:card_ZFC}

Throughout, we fix $(e_n)$ an $F$-basis for $E$ (e.g., $E=\bigoplus_{n\in\omega} F$ and $e_n$ is the $n$th unit coordinate vector). If $X$ is a subset of $E$, or a sequence of vectors in $E$, we write $\langle X\rangle$ for its linear span. 

For $x\in E$, the \emph{support} of $x$ is given by
\[
	\supp(x) = \{n\in\omega:x=\sum a_ie_i \Rightarrow a_n\neq 0\}.
\]
For nonzero vectors $x,y\in E$ and $M\in\omega$, we write $x>M$ if $\min(\supp(x))>M$, and $x<y$ if $\max(\supp(x))<\min(\supp(y))$. We call a sequence of nonzero vectors $(x_n)$ a \emph{block sequence} if $x_n<x_{n+1}$ for all $n$. A space spanned by a block sequence is a \emph{block subspace}. 

As mentioned in \S\ref{sec:Intro}, every infinite-dimensional subspace of $E$ contains a block subspace (Lemma 2.1 in \cite{MR3864398}), and the block sequence forming the basis of a block subspace is unique, up to scaling. Unless otherwise specified, a block subspace is always assumed to be infinite-dimensional.

We begin with the following easy facts:

\begin{prop}
	Every almost disjoint family of (block) subspaces is contained in a mad family of (block) subspaces.	
\end{prop}

\begin{proof}
	This is a standard Zorn's Lemma argument.
\end{proof}

\begin{prop}
	There is an almost disjoint family of block subspaces, and thus a mad family of block subspaces, of size $\cc$.	
\end{prop}

\begin{proof}
	Let $\LA$ be an almost disjoint family on $\omega$ of size $\cc$, as in (\ref{ex:perfect_ad}) above. The image of $\LA$ under the injective map $x\mapsto\langle(e_n)_{n\in x}\rangle$ is easily seen to be an almost disjoint family of subspaces.
\end{proof}

%Note that any nontrivial almost disjoint family of subspaces contained in the image of the ``diagonal map'' $x\mapsto\langle(e_n)_{n\in x}\rangle$ used above fails to be maximal: $\langle(e_{2n}+e_{2n+1})_{n\in\omega}\rangle$ will be disjoint from every subspace having infinite codimension in this image.

%To deal with general subspaces, the following definition will be useful:
%
%\begin{defn}
%	A sequence $(x_n)$ of nonzero vectors in $E$ is in \emph{reduced echelon form} if the matrix whose $n$th row is given by $x_n$, expressed with respect to the basis $(e_n)$, is in reduced echelon form.
%\end{defn}

%As all vectors have finite support, this definition is unambiguous even for infinite sequences. Note that row reduction of an infinite matrix with finitely-supported rows will always converge coordinatewise to an infinite reduced echelon form matrix. It follows that every subspace has a (unique) basis in reduced echelon form, and by passing to a sufficiently ``spread out'' subsequence, that every infinite-dimensional subspace contains an infinite-dimensional block subspace.%Consequently, an almost disjoint family of block subspaces which is maximal amongst block subspaces is maximal.\footnote{Using the terminology of forcing, we are just using the fact that a maximal antichain in a dense suborder a preorder is maximal in the preorder itself.}

Given an infinite-dimensional subspace $Y$ and an $M\in\omega$, we write $Y/M$ for the set of $y\in Y$ with $y>M$; $Y/M$ is always an infinite-dimensional subspace of $Y$. Given a vector $x$, we write $Y/x$ for $Y/\max(\supp(x))$. The following lemma will be key to much of what follows.%The following key lemma says that, given a finite block sequence and a subspace $Y$, we can always go far enough out (in terms of support) so that adding a new vector to the block sequence changes its intersection with $Y$ in as minimal a way as possible.

\begin{lemma}\label{lem:extend}
	Let $Y$ be a block\footnote{In earlier versions of this paper, including in \cite{Smythe_thesis}, this lemma was stated for \emph{arbitrary} infinite-dimensional subspaces $Y$, rather than block subspaces. This was in error, as the following counterexample shows: Let $Y=\langle e_0+e_2,e_1+e_3,\ldots,e_{2n}+e_{2n+2},e_{2n+1}+e_{2n+3},\ldots\rangle$. Note that $Y$ contains either $e_0+e_{2n+2}$ or $e_0-e_{2n+2}$, for each $n$. Thus, if we take $x_0=e_0$, then for any $M$, we can find an $n$ with $2n+2>M$ and so $e_0\in\langle e_0,e_{2n+2}\rangle\cap Y\neq\{0\}=\langle e_0\rangle\cap Y$. This is related to the fact that the basis for $Y$ cannot put in a ``row reduced echelon form'', and appears to be the essential difficulty in removing ``block'' from many of the arguments herein.} subspace of $E$ and $x_0,\ldots,x_m$ nonzero vectors in $E$. 
	\begin{enumerate}[label=\textup{(\alph*)}]
		\item If $x\in Y$, then
		\[
			\langle x_0,\ldots,x_m,x\rangle\cap Y=\langle x_0,\ldots,x_m\rangle\cap Y+\langle x\rangle.
		\]	
		\item There is an $M\in\omega$ (that depends only on $Y$ and $\max(\bigcup_{i=0}^m\supp(x_i))$) such that whenever $x>M$ and $x\notin Y$,
		\[
			\langle x_0,\ldots,x_m,x\rangle\cap Y=\langle x_0,\ldots,x_m\rangle\cap Y.
		\]
	\end{enumerate}
\end{lemma}

\begin{proof}
	(a) is just the special case of the modularity law for subspaces (and holds for arbitrary $Y$):
	\[
		Z\subseteq Y \quad\text{implies}\quad (X+Z)\cap Y=(X\cap Y)+Z,
	\]
	where $X=\langle x_0,\ldots,x_m\rangle$ and $Z=\langle x\rangle$.
	
	(b) Suppose that $Y=\langle (y_n)\rangle$, where $(y_n)$ is a block sequence. Put $K=\max(\bigcup_{i=0}^m\supp(x_i))$ and let $N$ be the largest index such that $\supp(y_N)\cap[0,K]\neq\emptyset$. Set	
	\[
		M=\max\left\{\max(\supp(y_N)),K\right\}.
	\]
	Take $x>M$ with $x\notin Y$.  Suppose that
	\[
		v=\lambda_0x_0+\cdots+\lambda_m x_m+\lambda x\in Y,
	\]
	with $\lambda_i$'s not all $0$. To prove the result, it suffices to show that $\lambda=0$. Towards a contradiction, suppose that $\lambda\neq 0$ and write
	\[
		\alpha_0y_0+\cdots+\alpha_k y_k=\lambda_0x_0+\cdots+\lambda_m x_m+\lambda x,
	\]
	for some $k\in\omega$. Since $x>M$, we must have that $k>N$. But, by our choice of $N$ and the fact that the $y_n$ are in block position, we have that
	\[
		\alpha_0y_0+\cdots+\alpha_N y_N=\lambda_0x_0+\cdots+\lambda_m x_m,
	\]
	which implies $x=\frac{1}{\lambda}(\alpha_{N+1}y_{N+1}+\cdots+\alpha_k y_k)\in Y$, a contradiction.
%	(b) Let $(y_n)$ be a basis for $Y$ in reduced echelon form and put $K=\max(\bigcup_{i=0}^m\supp(x_i))$. Let $N$ be the largest index such that $\supp(y_N)\cap[0,K]\neq\emptyset$, which exists since the $y_n$ are linearly independent. Let
%	\[
%		M=\max\left\{\max\left(\bigcup_{n=0}^{N}\supp(y_n)\right),K\right\}.
%	\]
%	Take $x>M$ with $x\notin Y$. Suppose that
%	\[
%		v=\lambda_0x_0+\cdots+\lambda_m x_m+\lambda x\in Y,
%	\]
%	with $\lambda_i$'s not all $0$.	To prove the result, it suffices to show that $\lambda=0$. Towards a contradiction, suppose that $\lambda\neq 0$ and write
%	\[
%		\alpha_0y_0+\cdots+\alpha_k y_k=\lambda_0x_0+\cdots+\lambda_m x_m+\lambda x,
%	\]
%	for some $k\in\omega$. Since $x>M$, we must have that $k>N$. We claim that
%	\[
%		\alpha_0y_0+\cdots+\alpha_N y_N=\lambda_0x_0+\cdots+\lambda_m x_m,
%	\]
%	which implies $x=\frac{1}{\lambda}(\alpha_{N+1}x_{N+1}+\cdots+\alpha_k y_k)\in Y$, a contradiction.
%	
%	Suppose that there is an $\ell>N$ such that $\alpha_\ell\neq 0$ and $\supp(y_\ell)\cap\supp(y_j)\neq\emptyset$ for some $j\leq N$. Since the $y_n$ are in reduced echelon form, the leading coefficient (when expressed with respect to $(e_n)$) of $\alpha_\ell y_\ell$ occurs in $v$, while being both below $x$ and above $x_m$, which is absurd. 
%	
%	Thus, for every $\ell>N$, either $\alpha_\ell=0$ or $y_\ell$ has support disjoint from $\bigcup_{n\leq N}\supp(y_n)$. Then, by our choice of $N$, 
%	\[
%		\alpha_0y_0+\cdots+\alpha_N y_N=\lambda_0x_0+\cdots+\lambda_m x_m,
%	\]
%	as claimed.
\end{proof}

\begin{lemma}\label{lem:disj_subs2}
	Suppose that $Y_0,\ldots, Y_n,Y_{n+1}$ are pairwise disjoint block subspaces and $x_0,\ldots,x_m$ nonzero vectors such that $\langle x_0,\ldots,x_m\rangle\cap Y_k=\{0\}$ for $k\leq n+1$. Then, there is an $x>x_m$ such that $\langle x_0,\ldots,x_m,x\rangle\cap Y_k=\{0\}$ for $k\leq n+1$.	
\end{lemma}

\begin{proof}
	By repeatedly applying Lemma \ref{lem:extend}(b), we can obtain an $M_0\geq\max(\bigcup_{i=0}^m\supp(x_m))$ such that whenever $x>M_0$ and not in any of the $Y_k$'s, $\langle x_0,\ldots,x_m,x\rangle\cap Y_k=\{0\}$ for $k\leq n+1$.
	
	To find such an $x$, start by picking $x_0'\in Y_0/M_0$, so $\langle x_0'\rangle\cap Y_0=\langle x_0'\rangle$ and $\langle x_0'\rangle\cap Y_k=\{0\}$ for $0<k\leq n+1$. By repeatedly applying Lemma \ref{lem:extend}, we can obtain an $M_1\geq M_0$ such that whenever $y\in Y_1/M_1$, we have that
	\begin{align*}
		\langle x_0',y\rangle\cap Y_0&=\langle x_0'\rangle\cap Y_0=\langle x_0'\rangle\\
		\langle x_0',y\rangle\cap Y_1&=\langle x_0'\rangle \cap Y_1+\langle y\rangle=\langle y\rangle\\
		\langle x_0',y\rangle\cap Y_k&=\langle x_0'\rangle\cap Y_k=\{0\} \quad\text{for $1<k\leq n+1$}.
	\end{align*}
	Pick $x_1'\in Y_1/M_1$. Continue in this fashion, using Lemma \ref{lem:extend} to choose an $M_\ell\geq M_{\ell-1}$ and $x_\ell'\in Y_\ell/M_\ell$, for $1\leq\ell\leq n+1$, so that
	\begin{align*}
		\langle x_0',\ldots,x_{\ell-1}',x_\ell'\rangle\cap Y_0&=\langle x_0',\ldots,x_{\ell-1}'\rangle\cap Y_0=\langle x_0'\rangle\\	
		%\langle x_0',x_1',\ldots,x_{\ell-1}',x_\ell'\rangle\cap Y_1&=\langle x_0',x_1',\ldots,x_{\ell-1}'\rangle\cap Y_1=\langle x_1'\rangle\\
		\vdots \\
		\langle x_0',\ldots,x_{\ell-1}',x_\ell'\rangle\cap Y_{\ell-1}&=\langle x_0',\ldots,x_{\ell-1}'\rangle\cap Y_{\ell-1}=\langle x_{\ell-1}'\rangle\\
		\langle x_0',\ldots,x_{\ell-1}',x_\ell'\rangle\cap Y_\ell &=\langle x_0',\ldots,x_{\ell-1}'\rangle \cap Y_\ell+\langle x_\ell'\rangle =\langle x_\ell'\rangle\\
		\langle x_0',\ldots,x_{\ell-1}',x_\ell'\rangle\cap Y_k&=\langle x_0',\ldots,x_{\ell-1}'\rangle\cap Y_k=\{0\} \quad\text{for $\ell<k\leq n+1$}.
	\end{align*}
	Then, $x=x_0'+\cdots+x_{n+1}'$ is not in any of the $Y_k$'s, and so $\langle x_0,\ldots,x_m,x\rangle\cap Y_k=\{0\}$ for $k\leq n+1$.
	%To find such an $x$, one can use Lemma \ref{lem:disj_subs1} repeatedly to build $x_0'<\cdots<x_{n+1}'$ above $M$ and satisfying $\langle x_0',\ldots,x_n'\rangle\cap Y_k=\langle x_k'\rangle$ for $k\leq n+1$. Then, $x=x_0'+\cdots+x_{n+1}'$ is not in $Y_k$ for $k\leq n+1$ and is as desired.
\end{proof}

If $X$ is a finite-codimensional subspace, then $\{X\}$ is always a mad family of subspaces. These are the only countable mad families of block subspaces.%One may think of this lack of finite mad families to be another instance of the failure of a kind of pigeonhole principle.

\begin{prop}\label{prop:mad_unctbl}
	Let $\LA$ be a maximal almost disjoint family of block subspaces of size $\geq 2$. Then, $\LA$ is uncountable.%\footnote{When $\LA$ is finite, this is a special case of the fact that an infinite-dimensional vector space cannot be written as a finite union of proper subspaces. This can be proved using Lemma \ref{lem:extend}.}
\end{prop}

\begin{proof}
	Suppose first that $\LA=\{Y_0,\ldots,Y_n,Y_{n+1}\}$ is a finite almost disjoint family of block subspaces. By replacing each $Y_k$ with a relatively finite-codimensional subspace, we may assume that they are pairwise disjoint. Pick an $x_0$ not in any of the $Y_k$'s, which can be done as in the proof of Lemma \ref{lem:disj_subs2}. By repeatedly applying Lemma \ref{lem:disj_subs2}, we can build an infinite block sequence $(x_m)$ such that for each $m$ and $k\leq n+1$, $\langle x_0,\ldots,x_m\rangle\cap Y_k=\{0\}$. Then, $\langle (x_m)\rangle$ is disjoint from each $Y_k$, witnessing that $\LA$ fails to be maximal.
	
	Suppose that $\LA=\{Y_n:n\in\omega\}$ is a countably infinite almost disjoint family of block subspaces. Again, by passing to finite-codimensional subspaces, we may assume that the $Y_k$ are pairwise disjoint. Pick a nonzero $x_0\in Y_0$. By Lemma \ref{lem:extend}, we can pick $x_1\in Y_1/x_0$ such that
	\begin{align*}
		\langle x_0,x_1\rangle \cap Y_0&=\langle x_0\rangle\\
		\langle x_0,x_1\rangle \cap Y_k&\subseteq \langle x_0,x_1\rangle \quad\text{for $k\geq 1$}.
	\end{align*}
	In general, given $x_0,\ldots, x_m$, we can apply Lemma \ref{lem:extend} to obtain $x_{m+1}\in Y_{m+1}/x_m$ such that
	 \begin{align*}
		\langle x_0,\ldots,x_m,x_{m+1}\rangle \cap Y_0&=	\langle x_0,\ldots,x_m\rangle \cap Y_0=\langle x_0\rangle\\
		\vdots\\
		\langle x_0,\ldots,x_m,x_{m+1}\rangle \cap Y_m &= \langle x_0,\ldots,x_m\rangle \cap Y_m\subseteq\langle x_0,\ldots,x_m\rangle\\
		\langle x_0,\ldots,x_m,x_{m+1}\rangle \cap Y_k&\subseteq \langle x_0,\ldots,x_m,x_{m+1}\rangle \quad\text{for $k\geq m+1$}.
	 \end{align*}
	Thus, $(x_m)$ is an infinite block sequence such that $\langle(x_m)\rangle\cap Y_n\subseteq\langle x_0,\ldots,x_n\rangle$ for each $n\in\omega$, and so again, $\LA$ fails to be maximal.
\end{proof}

%In the remainder of this section, following \S3 of \cite{MR3685044}, we will establish more refined lower bounds for $\a_{\mathrm{vec},F}$.

For $f,g\in\omega^\omega$, we write $f<^*g$ if there is some $N$ such that $f(n)<g(n)$ for all $n\geq N$. A family of functions $\LB\subseteq\omega^\omega$ is \emph{bounded} if there is some $h\in\omega^\omega$ such that $f<^* h$ for all $f\in\LB$, and \emph{unbounded} otherwise. We write 
\[
	\b=\min\{|\LB|:\LB\text{ is an unbounded family in $\omega^\omega$}\}.
\]
It is easy show that $\b$ is uncountable and it is well-known that $\b\leq\a$ (see Proposition 8.4 in \cite{MR2768685} or Theorem 3.1 in \cite{MR776622}). The corresponding result for infinite-dimensional block subspaces of $\FIN$ was proved in \cite{MR3685044}, however their proof does not appear to easily generalize; our proof here uses Lemma \ref{lem:extend} to adapt the usual proof of $\b\leq\a$.

\begin{prop}\label{prop:bleqa}
	$\b\leq\a_{\mathrm{vec},F}$.
\end{prop}

\begin{proof}
	Let $\LA$ be an infinite almost disjoint family of block subspaces with $|\LA|=\kappa<\b$. We may enumerate $\LA$ as $\{Y_\alpha:\alpha<\kappa\}$. By passing to finite-codimensional subspaces, we may assume that the $Y_n$, for $n<\omega$, are pairwise disjoint. For $\omega\leq\alpha<\kappa$, define $f_\alpha$ by
	\[
		f_\alpha(n)=\min\{k: Y_\alpha\cap Y_n\subseteq\langle e_0,\ldots,e_k\rangle\}.
	\]
	Define $f_m$ for $m<\omega$ arbitrarily. For each $\alpha<\kappa$, let $g_\alpha\in\omega^\omega$ be such that whenever $y_0<\cdots<y_k$ are such that $\supp(y_k)\subseteq[0,n]$ and $x>g_\alpha(n)$,
	\[
		\langle y_0,\ldots,y_k,x\rangle\cap Y_\alpha=\begin{cases} \langle y_0,\ldots,y_k\rangle\cap Y_\alpha &\text{if $x\notin Y_\alpha$,}\\ \langle y_0,\ldots,y_k\rangle\cap Y_\alpha+\langle x\rangle &\text{if $x\in Y_\alpha$.}\end{cases}
	\] 
	Such functions exist by Lemma \ref{lem:extend} (we are using the fact that the $M$ in Lemma \ref{lem:extend}(b) depends only on the given subspace and the maximum of the \emph{supports} of the given finite sequence). As $\kappa<\b$, there is an $h\in\omega^\omega$, which we may take strictly increasing, with $\max\{f_\alpha,g_\alpha\}<^*h$ for all $\alpha<\kappa$. Define a block sequence $X=(x_n)$ by choosing $x_0\in Y_0$ and $x_{n+1}\in Y_{n+1}/h(\max(\supp(x_n)))$ for all $n\in\omega$. We claim that $\langle X\rangle$ is almost disjoint from each $Y_\alpha$.
	
	Case 1: $\alpha=m<\omega$. Let $N> m$ be such that $g_m(n)<h(n)$ for all $n\geq N$. Note that $\max(\supp(x_N))\geq N$. If $k\geq N$, then $x_{k+1}\in Y_{k+1}/g_m(\max(\supp(x_{k}))$ and, since $Y_{k+1}$ and $Y_m$ are disjoint,
	\[
		\langle x_N,\ldots,x_{k},x_{k+1}\rangle\cap Y_m=\langle x_N,\ldots,x_{k}\rangle\cap Y_m=\cdots=\langle x_N\rangle\cap Y_m=\{0\}.
	\]
	This shows that $\langle X/x_{N-1}\rangle$ is disjoint from $Y_m$.
	
	Case 2: $\omega\leq\alpha<\kappa$. Let $N$ be such that $\max\{f_\alpha(n),g_\alpha(n)\}<h(n)$ for all $n\geq N$. Again, note that $\max(\supp(x_N))\geq N$. If $k\geq N$, then $x_{k+1}\in Y_{k+1}/g_\alpha(\max(\supp(x_k)))$, so
	\[
		\langle x_N,\ldots,x_k,x_{k+1}\rangle\cap Y_\alpha=\begin{cases} \langle x_N,\ldots,x_k\rangle\cap Y_\alpha &\text{if $x_{k+1}\notin Y_\alpha$,}\\ \langle x_N,\ldots,x_k\rangle\cap Y_\alpha+\langle x_{k+1}\rangle &\text{if $x_{k+1}\in Y_\alpha$.}\end{cases}
	\]
	However, as $x_{k+1}>f_\alpha(k+1)$ and $x_{k+1}\in Y_{k+1}$, it must be that $x_{k+1}\notin Y_\alpha$. Then, as in Case 1,
	\[
		\langle x_N,\ldots,x_{k},x_{k+1}\rangle\cap Y_\alpha=\langle x_N,\ldots,x_{k}\rangle\cap Y_\alpha=\cdots=\langle x_N\rangle\cap Y_\alpha=\{0\},
	\]
	showing, again, that $\langle X/x_{N-1}\rangle$ is disjoint from $Y_\alpha$. 
	
	Thus, $\LA$ fails to be maximal, and so $\b\leq\a_{\mathrm{vec},F}$.
\end{proof}

Recall that $\FIN$ is the collection of all nonempty finite subsets of $\omega$. For $a,b\in\FIN$, we write $a<b$ if $\max(a)<\min(b)$, and call a sequence $(a_n)$ of elements of $\FIN$ a \emph{block sequence} if $a_n<a_{n+1}$ for all $n\in\omega$. Let $\FIN^{[\infty]}$ denote the set of infinite block sequences in $\FIN$. For $A=(a_n)\in\FIN^{[\infty]}$,
	\[
		\FU(A)=\{a_{n_0}\cup\cdots\cup a_{n_k}:n_0<\cdots<n_k\},
	\] 
	is the \emph{combinatorial subspace} generated by $A$. We say that $A,B\in\FIN^{[\infty]}$ are \emph{almost disjoint} if $\FU(A)\cap\FU(B)$ is finite. Following Brendle and Garc\'ia \'Avila \cite{MR3685044}, let $\a_{\FIN}$ be the minimum cardinality of an infinite maximal almost disjoint family (defined in the obvious way) of block sequences in $\FIN$. As commented in \S\ref{sec:Intro}, this is the same as $\a_{\mathrm{vec},F}$ when $|F|=2$.

We denote by $\mathrm{non}(\LM)$ the minimum size of a nonmeager subset of $\R$. Brendle and Garc\'ia \'Avila show that $\mathrm{non}(\LM)\leq\a_{\FIN}$ (Theorem 3 in \cite{MR3685044}) by showing $\b\leq\a_{\FIN}$ (Proposition 12 in \cite{MR3685044}), $\mathrm{non}(\LM)=\max\{\b,\b(pbd\neq^*)\}$ (Lemma 15 in \cite{MR3685044}, attributed to folklore), and finally, $\b(pbd\neq^*)\leq\a_{\FIN}$ (Theorem 16 in \cite{MR3685044}). Here, $\b(pbd\neq^*)$ is the common (Lemma 14 in \cite{MR3685044}) value of the cardinals $\b_h(p\neq^*)$, where, for $h:\omega\to\omega$ a function with $h(n)\to\infty$ as $n\to\infty$, $\b_h(p\neq^*)$ is the minimum size of a family $\LF\subseteq\omega^\omega$ such that for all partial $g:\omega\rightharpoonup\omega$ with infinite domain and bounded by $h$ on that domain, there is an  $f\in \LF$ which is equal to $g$ infinitely often.

For $A=(a_n)\in\FIN^{[\infty]}$, denote by
\[
	E_A=\bigcup\{a_n:|a_n|=1\}.
\]
A careful reading of their proof reveals that Brendle and Garc\'ia \'Avila have shown the following:

\begin{thm}[cf.~Theorem 16 in \cite{MR3685044}]\label{thm:B-GA}
	Suppose that $\LA\subseteq\FIN^{[\infty]}$ satisfies the following for all $A,A'\in\LA$:
	\begin{enumerate}[label=\textup{(\roman*)}]
		\item $E_A$ is coinfinite, and
		\item if $A\neq A'$, then $E_A\cap E_{A'}$ is finite.
	\end{enumerate}
	Then, if $\omega\leq|\LA|<\b(pbd\neq^*)$, there is a $B\in\FIN^{[\infty]}$ which is almost disjoint from each element of $\LA$.
\end{thm}

For $X=(x_n)$ a block sequence in $E$, let $\supp(X)=(\supp(x_n))\in\FIN^{[\infty]}$. If $\LA$ is a collection of infinite-dimensional block subspaces of $E$, then let 
\[
	\supp(\LA)=\{\supp(X):\text{$X$ is a block sequence and }\langle X\rangle\in\LA\}.
\]
Note if $X$ and $Y$ are block sequences spanning the same subspace, then $\supp(X)=\supp(Y)$. The proof of the following is easy and omitted.

\begin{lemma}\label{lem:supp_proj}
	For any block sequence $X$ in $E$, if $A\in\FIN^{[\infty]}$ is such that $\FU(A)\subseteq\FU(\supp(X))$, then there is a block sequence $Y$ in $E$ with $\langle Y\rangle\subseteq\langle X\rangle$ and $\supp(Y)=A$.\footnote{This lemma implies that the $\supp$ map is a \emph{projection}, in the sense of forcing, between block sequences in $E$ and those in $\FIN$. See the related discussion in \S 6 of \cite{MR3864398}.}	\qed
\end{lemma}

\begin{lemma}\label{lem:supp_ad}
	If $\LA$ is a family of infinite-dimensional block subspaces of $E$ and $A\in\FIN^{[\infty]}$ is almost disjoint (in the sense of $\FIN$) from every element of $\supp(\LA)$, then for any block sequence $X$ in $E$ with $\supp(X)=A$, $\langle X\rangle$ will be almost disjoint (in the sense of $E$) from every $Y\in\LA$.
\end{lemma}

\begin{proof}
	Let $\LA$ and $A$ be as described, and suppose that there is some block sequence $X$ with $\supp(X)=A$, and a subspace in $\LA$, with block basis $Y$, such that $\langle Y\rangle\cap\langle X\rangle$ is infinite-dimensional. Let $Z$ be an infinite block sequence in $\langle Y\rangle\cap\langle X\rangle$. Then, $\supp(Z)$ will witness that $A$ fails to be almost disjoint from $\supp(Y)$.
\end{proof}

\begin{lemma}\label{lem:supp_B-GA}
	If $\LB$ is an infinite almost disjoint family of block subspaces of $E$, then $\LA=\supp(\LB)$ satisfies conditions (i) and (ii) in Theorem \ref{thm:B-GA}.
\end{lemma}

\begin{proof}
	This follows immediately from the observation that if $A=\supp(X)$ for $X$ a block sequence in $E$, and $n\in E_A$, then $e_n\in\langle X\rangle$.
\end{proof}

Putting Lemma \ref{lem:supp_proj}, \ref{lem:supp_ad} and \ref{lem:supp_B-GA} together with Proposition \ref{prop:bleqa} and Theorem \ref{thm:B-GA}, we have:

\begin{cor}\label{cor:nonMleqa_vec}
	$\mathrm{non}(\LM)\leq\a_{\mathrm{vec},F}$.\qed%$\mathrm{non}(\LM)$ is less than or equal to the minimum cardinality of an infinite mad family of block subspaces.\qed
\end{cor}

\section{Cardinality: Consistency results}\label{sec:card_cons}

It follows from Proposition \ref{prop:mad_unctbl} that under $\CH$, every mad family of subspaces is of size $\cc$. Likewise, since $\MA_\kappa(\sigma\text{-centered})$ implies $\kappa<\p$ (cf.~\cite{MR643555}), and $\p\leq\b$ (see \cite{MR776622}), these together with Proposition \ref{prop:bleqa} yield $\kappa<\a_{\mathrm{vec},F}$. We give here a direct proof of this fact:

%Likewise, Proposition \ref{prop:bleqa} and the fact that $\MA$ implies $\b=\c$, shows that $\MA$ implies this as well. We give here a direct proof:

\begin{thm}\label{thm:MA_a_large}
	\textup{($\MA_\kappa(\sigma\text{-centered})$)} $\kappa<\a_{\mathrm{vec},F}$.	
\end{thm}

\begin{proof}
	Let $\LA$ be an infinite almost disjoint family of block subspaces. Define a poset $\P$ to be all pairs $(s,F)$ where $s$ is a finite normalized (i.e., leading coefficients are equal to $1$) block sequence in $E$ and $F$ a finite subset of $\LA$. We order elements of $\P$ by $(s',F')\leq(s,F)$ if $s'\sqsupseteq s$, $F'\supseteq F$, and $\forall X\in F(\langle s'\rangle\cap X\subseteq\langle s\rangle)$. Note that if $(s,F'),(s,F)\in\P$, for a fixed $s$, then $(s,F'\cup F)\in\P$ and extends both conditions. As there are only countably many such $s$, this shows that $\P$ is $\sigma$-centered. If $G$ is a filter in $\P$, then we let $X_G=\langle\bigcup\{s:\exists F((s,F)\in G)\}\rangle$.
	
	Observe that if $X\in\LA$, then the set $D_X=\{(s,F)\in\P:X\in F\}$ is dense, and if $G$ is a filter in $\P$ with $G\cap D_X\neq\emptyset$, then $X_G\cap X$ is finite dimensional. For $n\in\omega$, let $E_n=\{(s,F)\in\P:|s|\geq n\}$. In order to see that the sets $E_n$ are dense, it suffices to show that a given $(s,F)$ in $\P$ can be extended to an $(s\concat x,F)$ in $\P$. This can be accomplished by using Lemma \ref{lem:extend} to obtain an $M$ for which whenever $x>M$ and not in $\bigcup F$, $\langle s\concat x\rangle\cap X=\langle s\rangle\cap X$ for each of the finitely many $X\in F$. Then, for any such $x$, $(s\concat x,F)\leq(s,F)$.

	If $|\LA|\leq\kappa$, by $\MA_\kappa(\sigma\text{-centered})$, there is a filter $G\subseteq\P$ which meets the sets $D_X$ and $E_n$, for $X\in\LA$ and $n\in\omega$. Then, $X_G$ witnesses that $\LA$ fails to be maximal.
\end{proof}

Let $\B_\kappa$ be \emph{$\kappa$-random forcing}, the set of all positive measure Borel subsets of $2^\kappa$ ordered by containment modulo null sets, where $\kappa\geq\omega$ and $2^\kappa$ is given the product measure. By the \emph{random model}, we mean the generic extension of a model of $\CH$ obtained by forcing with $\B_{\aleph_2}$. It is well-known that in the random model, $\b=\d=\a=\aleph_1$ and $\mathrm{non}(\LM)=\c=\aleph_2$ (see, e.g., \S 11.4 of \cite{MR2768685}). Thus, by Corollary \ref{cor:nonMleqa_vec}, we have:

\begin{cor}\label{cor:a<a_vec}
	In the random model, $\aleph_1=\a<\a_{\mathrm{vec},F}=\aleph_2$.\qed
\end{cor}

Let $\C_\kappa$ be \emph{$\kappa$-Cohen forcing}, the set of all finite partial functions $p$ with $\dom(p)\subseteq\kappa\times\omega$ and $\ran(p)\subseteq 2$, ordered by extension. We identify $\C_{\aleph_0}$ with the set $\C$ of all finite partial functions $p$ with $\dom(p)\subseteq\omega$ and $\ran(p)\subseteq 2$. By the \emph{Cohen model}, we mean the generic extension of a model of $\CH$ obtained by forcing with $\C_{\aleph_2}$. Theorem \ref{thm:cohen} is stated as Theorem 3.7 in \cite{MR1756262}, however the proof given is just a reference to \cite{MR597342}. We give a complete proof here. See also Theorem 4 in \cite{MR3685044} for the analogous result for $\FIN$. %The following result for arbitrary subspaces of $E$ is due to Kolman (Theorem 3.7 in \cite{MR1756262}), however as only a sketch of the proof is given there, we give a complete proof here.

\begin{thm}\label{thm:cohen}
	In the Cohen model, $\aleph_1=\a_{\mathrm{vec},F}<\c$.
\end{thm}

\begin{proof}
	We follow the proof of the corresponding result for mad families of subsets of $\omega$, Theorem 2.3 in Ch.~VIII of \cite{MR597342}. We define a maximal almost disjoint family $\LA=\{X_\xi:\xi<\omega_1\}$ of block subspaces having the property that it remain maximal after forcing with $\C$. By standard properties of Cohen forcing (Lemma 2.2 in Ch.~VIII of \cite{MR597342}), this suffices.
	
	Using $\CH$ in the ground model, let $(p_\xi,\tau_\xi)$ for $\omega\leq\xi<\omega_1$ enumerate all pairs $(p,\tau)$ such that $p\in\C$ and $\tau$ is a nice $\C$-name for a subset of $E$ (in the sense of Definition 5.11 in Ch.~VII of \cite{MR597342}). We recursively pick block subspaces $X_\xi$ as follows: Let $X_n$, $n<\omega$, be any sequence of infinite-dimensional almost disjoint block subspaces. If $\omega\leq\xi<\omega_1$, and we have chosen $X_\eta$ for all $\eta<\xi$, choose $X_\xi$ almost disjoint from each of the (countably many) $X_\eta$ for $\eta<\xi$ and so that if 
	\begin{align}\label{cohen_ind_hyp}
		p_\xi\forces_\C\text{$\tau_\xi$ is an infinite-dimensional subspace and $\forall\eta<\xi\dim(\tau_\xi\cap \check{X}_\eta)<\infty$}
	\end{align}
	then
	\[
		\forall n\forall q\leq p_\xi\exists r\leq q\exists v>n(v\in X_\xi \text{ and } r\forces_\C\check{v}\in\tau_\xi).
	\]
	To see that $X_\xi$ can be chosen, assume that (\ref{cohen_ind_hyp}) holds. Let $Y_i$ enumerate $\{X_\eta:\eta<\xi\}$ and let $q_i$ enumerate $\{q:q\leq p_\xi\}$. By (\ref{cohen_ind_hyp}), for each $i$, $q_i\forces_\C\dim(\tau_\xi\cap\check{Y}_i)<\infty$. We construct $r_i\in\C$ and $x_i\in E$ inductively in $i$. Pick $r_0\leq q_0$ and $x_0$ a nonzero vector so that $r_0\forces_\C \check{x}_0\in\tau_\xi$. Having chosen $r_0,\ldots,r_n$ and $x_0<\cdots<x_n$ so that $r_i\leq q_i$ and 
	\[
		r_i\forces_\C \check{x}_i\in\tau_\xi \land \forall k\leq i(\langle \check{x}_0,\ldots,\check{x}_i\rangle\cap \check{Y}_k\subseteq\langle\check{x}_0,\ldots,\check{x}_{k}\rangle),
	\]
	for each $i\leq n$, apply Lemma \ref{lem:extend} to find $r_{n+1}\leq q_{n+1}$ and $x_{n+1}>x_n$ so that 
	\[
		r_{n+1}\forces_\C \check{x}_{n+1}\in\tau_\xi \land \forall k\leq n+1(\langle \check{x}_0,\ldots,\check{x}_n,\check{x}_{n+1}\rangle\cap \check{Y}_k\subseteq\langle \check{x}_0,\ldots,\check{x}_{k}\rangle).
	\]
	This is done as in the infinite case of Proposition \ref{prop:mad_unctbl}. Let $X_\xi=\langle(x_n)\rangle$.
	
	Clearly $\LA$ is an almost disjoint family. It suffices to show that it is maximal in $\V[G]$, where $G$ is $\V$-generic for $\C$. Towards a contradiction, suppose that for some $(p_\xi,\tau_\xi)$ with $p_\xi\in G$,
	\[
	 	p_\xi\forces_\C\text{$\tau_\xi$ is an infinite-dimensional subspace and $\forall X\in\check{\LA}(\dim(\tau_\xi\cap X)<\infty)$}.
	\]
	In particular, (\ref{cohen_ind_hyp}) holds at $\xi$. But $p_\xi\forces_\C\dim(\tau_\xi\cap \check{X}_\xi)<\infty$, so there is a $q\leq p_\xi$ and an $N$ so that $q\forces_\C\tau_\xi\cap \check{X}_\xi\subseteq\langle \check{e}_0,\ldots,\check{e}_N\rangle$, contradicting that
	\[
		\exists r\leq q\exists x>N(x\in X_\xi \land r\forces_\C \check{x}\in\tau_\xi).\qedhere
	\]
\end{proof}

Given a notion of forcing $\P$, we say that a mad family of subspaces $\LA$ is \emph{$\P$-indestructible} if $\LA$ remains maximal after forcing with $\P$. The proof of Theorem \ref{thm:cohen} above shows that, assuming $\CH$, there is a $\C$-indestructible mad family of subspaces.

Let $\mathbb{S}$ be \emph{Sacks forcing}, the collection of all perfect subtrees of $2^{<\omega}$, ordered by inclusion. $\mathbb{S}$ enjoys the \emph{Sacks property} (cf.~Lemma 2.1 in \cite{MR556894}): whenever $p\in\mathbb{S}$ and $\dot{g}$ is an $\mathbb{S}$-name for an element of $\omega^\omega$, there is a $q\leq p$ and a function $F:\omega\to\LP(\omega)$ such that for all $n$, $|F(n)|\leq 2^n$ and $q\forces\forall n(\dot{g}(n)\in F(n))$. It follows that $\mathbb{S}$ is \emph{$\omega^\omega$-bounding}: every element of $\omega^\omega$ in the generic extension is bounded by some element of the ground model. We note that $\mathbb{S}$ is proper.\footnote{See, e.g., \cite{MooreYST} for more details on properness.}

\begin{thm}\label{thm:sacks_indes}
	\textup{($\CH$)} If $\P$ is a proper poset of size $\aleph_1$ having the Sacks property, then there is a $\P$-indestructible mad family of block subspaces.
\end{thm}

\begin{proof}
	Using $\CH$ and properness, we can construct a sequence of pairs $(p_\xi,\tau_\xi)$, $\xi<\omega_1$, so that:
	\begin{enumerate}[label=\textup{(\roman*)}]
		\item $\tau_{\xi}$ is a nice $\P$-name for an infinite block sequence in $E$, with all antichains occurring in $\tau_\xi$ countable, and
		\item $p_\xi\in\P$ is such that if there are $\tau$ and $p\in\P$ forcing that $\tau$ is an infinite block sequence, then there is a $\xi$ such that $p_\xi\leq p$ and $p_\xi\forces\tau =\tau_\xi$.
	\end{enumerate}
	
	We construct a family of block sequences $\LA=\{X_\alpha:\alpha<\omega_1\}$ recursively as follows: Begin by letting $\{X_i:i\in\omega\}$ be any almost disjoint family of block sequences (i.e., the corresponding subspaces are almost disjoint).
	
	At stage $\alpha\geq\omega$: If
	\[
		p_\alpha\not\forces\forall\xi<\alpha(\dim(\langle\tau_\alpha\rangle\cap\langle\check{X}_\xi\rangle)<\infty),
	\]
	then choose $X_\alpha$ to be any block sequence almost disjoint from all of the $X_\xi$ for $\xi<\alpha$. Otherwise, enumerate by $(\dot{v}_n)$ and $(\dot{I}_n)$ $\P$-names for vectors (in block position) and intervals containing their supports, respectively, which are forced by $p_\alpha$ to make up $\tau_\alpha$. Enumerate $\alpha$ as $(\xi_n)_{n<\omega}$.
	
	As the $X_{\xi_n}$ are almost disjoint, there is an $f\in\omega^\omega$ so that for all $n$, $X_{\xi_0}/f(0),\ldots,X_{\xi_n}/f(n)$ are disjoint. By our assumption on $p_\alpha$, there is a $\P$-name $\dot{g}$ for an element of $\omega^\omega$ so that
	\[
		p_\alpha\forces\forall n(\langle\tau_\alpha/\dot{g}(n)\rangle\cap\langle \check{X}_{\xi_n}\rangle=\{0\}).
	\]
	
	\begin{claim}
		If $Y_0,\ldots,Y_n,Y_{n+1}$ are disjoint block sequences and $x_0<\cdots<x_n$ so that for all $k\leq n$, $\langle x_0,\ldots,x_n\rangle\cap\langle Y_k\rangle=\{0\}$, then there is an $M$ so that whenever $x>M$ and not in any of $\langle Y_0\rangle,\ldots,\langle Y_n\rangle,\langle Y_{n+1}\rangle$, then for all $k\leq n+1$, $\langle x_0,\ldots,x_n,x\rangle\cap\langle Y_k\rangle=\{0\}$.
	\end{claim}
	
	\begin{proof}[Proof of claim.]
		See the proof of Lemma \ref{lem:disj_subs2}.
	\end{proof}
	
	By the claim, there is a $\P$-name $\dot{h}$ for an element of $\omega^\omega$ so that %$p_\alpha$ forces that ``if $i_0<\cdots<i_n$ and $\dot{h}(0)<\dot{v}_{i_0},\ldots,\dot{h}(n)<\dot{v}_{i_n}$, then $\forall k\leq n\langle \dot{v}_{i_0},\ldots,\dot{v}_{i_n}\rangle\cap\langle \check{X}_{\xi_k}/\check{f}(k)\rangle=\{0\})$''.
	\begin{align*}
		p_\alpha\forces &\forall n[(i_0<\cdots<i_n \text{ and } \dot{h}(0)<\dot{v}_{i_0},\ldots,\dot{h}(n)<\dot{v}_{i_n})\\
		&\Rightarrow \forall k\leq n\langle \dot{v}_{i_0},\ldots,\dot{v}_{i_n}\rangle\cap\langle \check{X}_{\xi_k}/\check{f}(k)\rangle=\{0\}].
	\end{align*}

	As $\P$ is $\omega^\omega$-bounding, there is a $p\leq p_\alpha$, and a function $m\in\omega^\omega$ so that
	\[
		p\forces\forall n(m(n)\geq\max\{\check{f}(n),\dot{g}(n),\dot{h}(n)\}),
	\]
	and so $p$ forces that $m$ shares the relevant properties of $f$, $\dot{g}$, and $\dot{h}$ above. Further, by $\omega^\omega$-bounding, there is an increasing sequence of intervals $(J_n)_{n<\omega}$, and a $p'\leq p$, so that
	\[
		p'\forces\forall n\exists m(\dot{I}_m\subseteq J_n).
	\]
	Choose a further increasing sequence of intervals $(K_n)_{n<\omega}$ so that $K_n$ contains at least $2^n$ many intervals of the form $J_m$, all of which are above $m(n)$.
	
	By the Sacks property, there is a $p''\leq p$ and a function $F$ with domain $\omega$ so that for each $n$, $|F(n)|\leq 2^n$ and each element of $F(n)$ is a collection of vectors in $E$, in block position, so that 
	\[
		p''\forces\forall n(\{\dot{v}_k:\dot{I}_k\subseteq \check{K}_n\}\in \check{F}(n)),
	\] 
	and for all $n$ and $A\in F(n)$, there is a $q\leq p''$ with
	\[
		q\forces\{\dot{v}_k:\dot{I}_k\subseteq \check{K}_n\}=\check{A}.
	\]
%	\begin{align*}
%		&p''\forces\forall n(\{\dot{v}_k:\dot{I}_k\subseteq \check{K}_n\}\in \check{F}(n)), \text{ and}\\
%		&\forall n\forall A\in F(n)\exists q\leq p''(q\forces\{\dot{v}_k:\dot{I}_k\subseteq \check{K}_n\}=\check{A}).
%	\end{align*}
	
	For each $n$, let $A_0,\ldots, A_{|F(n)|-1}$ enumerate $F(n)$. We pick vectors $u^0_n$ recursively as follows: Let $u^0_n$ be the first element of $A_0$. Having defined $u^0_n<\cdots<u^j_n$, with $u^i_n\in A_i$, choose $u^{j+1}_n$ to the first element of $A_{j+1}$ with support above $u^j_n$. Note that this can be done as each $A_k$ must contain elements with supports in each of $2^n$ distinct intervals $J_m$. Let $X_\alpha=(u_0^0,\ldots,u^{|F(0)|-1}_0,u_1^0,\ldots,u_1^{|F(1)|-1},\ldots)$. Observe that our choice of $m$ ensures that $X_\alpha$ is a block sequence and is almost disjoint from each $X_\xi$ for $\xi<\alpha$. That
	\[
		p''\forces\dim(\langle \tau_\alpha\rangle\cap\langle X_\alpha\rangle)=\infty
	\]
	is ensured by the construction. It is then easy to show that $\LA=\{\langle X_\alpha\rangle:\alpha<\omega_1\}$ is forced to be a mad family of subspaces by any condition in $\P$.
\end{proof}

By the \emph{Sacks model}, we mean the generic extension of a model of $\CH$ obtained by forcing with a countable support iteration of Sacks forcing of length $\omega_2$, see e.g., \cite{MR556894} or \cite{MR1900391}. Theorem \ref{thm:sacks} below is also a corollary of Theorem \ref{thm:cohen} and a general theorem of Zapletal (Theorem 0.2 in \cite{MR1978947}), though the latter makes use of large cardinals which are not necessary here.

\begin{thm}\label{thm:sacks}
	In the Sacks model, $\aleph_1=\a_{\mathrm{vec},F}<\c$.	
\end{thm}

\begin{proof}
	This is proved using Theorem \ref{thm:sacks_indes}, exactly as Theorem III.2 in \cite{MR1900391}, which the reader may consult for details.
\end{proof}

We note that it follows directly from Theorem \ref{thm:sacks_indes} that in the model obtained by forcing over a model of $\CH$ with the ``side-by-side''  (i.e., countable support product of) Sacks forcing \cite{MR794485} of length $\omega_2$, $\a_{\mathrm{vec},F}=\aleph_1$ as well. This is because any reals added in the side-by-side model are added by a product of $\omega_1$ many copies of Sacks forcing, which is proper, has the Sacks property, and preserves $\CH$ in the intermediate model.

Lastly, following \cite{MR0307913}, we turn to the problem of producing a ``large spectrum'' of cardinalities of mad families of subspaces. Given an uncountable regular cardinal $\kappa$, let
%\[
%	\LD_\kappa = \{(\alpha,\beta)\in\kappa\times\kappa:\alpha\text{ is a limit ordinal, }\cf(\alpha)>\omega,\text{ and }\beta<\alpha\}.
%\]
\[
	\LD_\kappa = \{(\alpha,\beta)\in\kappa\times\kappa:\alpha\text{ is an uncountable limit ordinal and }\beta<\alpha\}.
\]
Let $\Q_\kappa$ be the set of all functions $p:F_p\times n_p\to E$ where $F_p\in[\LD_\kappa]^{<\omega}$, $n_p\in\omega$, and for each $(\alpha,\beta)\in F_p$, $(p(\alpha,\beta,0),\ldots,p(\alpha,\beta,n_p-1))$ is a block sequence in $E$. We say $q\leq p$ if $q\supseteq p$ and whenever $(\alpha,\beta),(\alpha,\gamma)\in F_p$ with $\beta\neq\gamma$, we have that
\[
	\langle(q(\alpha,\beta,i))_{i<n_q}\rangle\cap\langle(q(\alpha,\gamma,i))_{i<n_q}\rangle=\langle(p(\alpha,\beta,j))_{j<n_p}\rangle\cap\langle(p(\alpha,\gamma,j))_{j<n_p}\rangle.
\]

%\begin{thm}\label{thm:hechler}
%	Let $\kappa$ be an uncountable regular cardinal. If $G$ is $\V$-generic for $\Q_\kappa$, then in $\V[G]$, for every cardinal $\lambda<\kappa$ of uncountable cofinality, there is a mad family of subspaces of $E$ of cardinality $\lambda$. In this model, $\cc\leq(\kappa^{\aleph_0})^\V$.
%\end{thm}
\begin{thm}\label{thm:hechler}
	Let $\kappa$ be an uncountable regular cardinal. If $G$ is $\V$-generic for $\Q_\kappa$, then in $\V[G]$, for every uncountable cardinal $\lambda<\kappa$ there is a mad family of block subspaces of $E$ of cardinality $\lambda$. In this model, $\kappa\leq\cc\leq(\kappa^{\aleph_0})^\V$.
\end{thm}

Typically, $\kappa=\kappa^{\aleph_0}$ and so $\cc=\kappa$ in the extension. Thus, it is consistent that $\cc>\aleph_2$ (or even $\cc>\aleph_{\omega_1}$, etc) and for every uncountable cardinal $\lambda\leq\cc$, there is a mad family of size $\lambda$. We will proceed with a series of lemmas.

\begin{lemma}
	$\Q_\kappa$ is ccc and $|\Q_\kappa|=\kappa$.	
\end{lemma}

\begin{proof}
	Suppose that $\{p_\xi:\xi<\aleph_1\}\subseteq\Q_\kappa$. By thinning down, we may assume that there is some fixed $n$ for which $n_{p_\xi}=n$ for all $\xi<\aleph_1$. By the $\Delta$-system lemma, we may further thin down so that the $F_{p_\xi}$ form a $\Delta$-system, that is, there is some finite set $R\subseteq\LD_\kappa$ for which $F_{p_\xi}\cap F_{p_\eta}=R$ for all $\xi\neq\eta<\aleph_1$. But as there are only countably many functions $R\times n\to E$, uncountably many of the $p_\xi$ agree on $R\times n$. Given such $p_\xi$ and $p_\eta$, it is then immediate that $q=p_\xi\cup p_\eta$ is a common extension. That $|\Q_\kappa|=\kappa$ is easy to check.	
\end{proof}

\begin{lemma}\label{lem:extend_F_p}
	Let $p\in\Q_\kappa$. For any $(\alpha,\beta)\in\LD_\kappa$, there is a $q\leq p$ with $(\alpha,\beta)\in F_q$.
\end{lemma}

\begin{proof}
	If $(\alpha,\beta)\notin F_p$, we can define $q\leq p$ so that $F_q=F_p\cup\{(\alpha,\beta)\}$, $n_q=n_p$, and $(q(\alpha,\beta,0),\ldots,q(\alpha,\beta,n_q-1))$ any block sequence in $E$ whatsoever.%Suppose that $(\alpha,\beta)\notin F_p$. Let $q$ be so that $F_q=F_p\cup\{(\alpha,\beta)\}$ and $n_q=n_p$, with $q'\restr(F_p\times n_p)=p$, and so that $(q(\alpha,\beta,0),\ldots,q(\alpha,\beta,n_q-1))$ is a block sequence whose span is disjoint from each of the finitely many, if any, block sequences $(p(\alpha,\eta,0),\ldots,p(\alpha,\eta,n_p-1))$ for $(\alpha,\eta)\in F_p$. Then $q$ is a condition extending $p$.	
\end{proof}

\begin{lemma}\label{lem:extend_n_p}
	Let $p\in\Q_\kappa$. For any $M>0$, there is a $q\leq p$ so that $n_q=n_p+1$ and $q(\alpha,\beta,n_p)>M$ for all $(\alpha,\beta)\in F_q$.
\end{lemma}

\begin{proof}
	Let $q(\alpha,\beta,i)=p(\alpha,\beta,i)$ for $i<n_p$ and $(\alpha,\beta)\in F_p$, as required.	Fix $\alpha$ occurring as a first coordinate in $F_p$. Enumerate by $\beta_0,\ldots,\beta_k$ those $\beta$ with $(\alpha,\beta)\in F_p$. Let $Y_j=\langle p(\alpha,\beta_j,0),\ldots,p(\alpha,\beta_j,n_p-1)\rangle$ for $j\leq k$. By repeated applications of Lemma \ref{lem:extend} (we are applying it to a finite-dimensional space $Y$, however the lemma remains true by essentially the same proof), there is an $N_0\geq M$ so that whenever $x>N_0$ and not in $Y_j$,
	\[
		\langle q(\alpha,\beta_0,0),\ldots,q(\alpha,\beta_0,n_p-1),x\rangle\cap Y_j=Y_0\cap Y_j,
	\]
	for $0<j\leq k$. Let $q(\alpha,\beta_0,n_p)$ be any vector $x>N_0$ and not in $\bigcup_{j\leq k}Y_j$. Let $Y_0'=\langle q(\alpha,\beta_0,0),\ldots,q(\alpha,\beta_0,n_p-1),q(\alpha,\beta_0,n_p)\rangle$.
%	
%	To define $q(\alpha,\beta_1,n_p)$, again by applications of Lemma \ref{lem:extend}, there is an $N_1\geq M$ so that whenever $x>N_1$ and not in $Y_0'$ or $Y_j$,
%	\[
%		\langle q(\alpha,\beta_1,0),\ldots,q(\alpha,\beta_1,n_p-1),x\rangle\cap Y_0'=Y_1\cap Y_0'=Y_1\cap Y_0,
%	\]
%	and
%	\[
%		\langle q(\alpha,\beta_1,0),\ldots,q(\alpha,\beta_1,n_p-1),x\rangle\cap Y_j=Y_1\cap Y_j
%	\]
%	for $1<j\leq k$. Let $q(\alpha,\beta_1,n_p)$ be any vector $x>N_1$ and not in $Y_0'\cup\bigcup_{1<j\leq k}Y_j$. Let $Y_1'=\langle q(\alpha,\beta_1,0),\ldots,q(\alpha,\beta_1,n_p-1),q(\alpha,\beta_1,n_p)\rangle$.
	
	Continue in this fashion, choosing $N_\ell\geq M$ so that whenever $x>N_\ell$ and not in $Y_i'$ or $Y_j$,
	\[
		\langle q(\alpha,\beta_\ell,0),\ldots,q(\alpha,\beta_\ell,n_p-1),x\rangle\cap Y_i'=Y_\ell\cap Y_i'=Y_\ell\cap Y_i,
	\]
	and
	\[
		\langle q(\alpha,\beta_\ell,0),\ldots,q(\alpha,\beta_\ell,n_p-1),x\rangle\cap Y_j=Y_\ell\cap Y_j,
	\]
	for $i<\ell$ and $\ell<j\leq k$. Let $q(\alpha,\beta_\ell,n_p)$ be any vector $x>N_\ell$ and not in $\bigcup_{i<\ell} Y_i'\cup\bigcup_{\ell<j\leq k}Y_j$. At the end of the construction, $q$ will be a condition with domain $F_p\times(n_p+1)$ extending $p$ and having the desired property.
\end{proof}

\begin{proof}[Proof of Theorem \ref{thm:hechler}.]
	Let $G$ be $\V$-generic for $\Q_\kappa$. By Lemmas \ref{lem:extend_F_p} and \ref{lem:extend_n_p}, $\bigcup G:\LD_\kappa\times\omega\to E$ so that for each $(\alpha,\beta)\in\LD_\kappa$, $G_{\alpha,\beta}(\cdot)=\bigcup G(\alpha,\beta,\cdot)$ is an infinite block sequence in $E$.

	Given an uncountable limit $\alpha<\kappa$, we claim that $\langle G_{\alpha,\beta}\rangle\cap\langle G_{\alpha,\gamma}\rangle$ is finite-dimensional, for $\beta\neq\gamma<\alpha$. Let $p\in\Q_\kappa$ be given with $(\alpha,\beta),(\alpha,\gamma)\in F_p$. By the definition of $\leq$ in $\Q_\kappa$, we have that
	\[
		p\forces\langle \dot{G}_{\alpha,\beta}\rangle\cap\langle \dot{G}_{\alpha,\gamma}\rangle=\langle(\check{p}(\alpha,\beta,i))_{j<n_p}\rangle\cap\langle(\check{p}(\alpha,\gamma,i))_{j<n_p}\rangle.
	\]
	Thus, $\langle G_{\alpha,\beta}\rangle\cap\langle G_{\alpha,\gamma}\rangle$ is forced to be finite-dimensional and $\LA_\alpha=\{\langle G_{\alpha,\beta}\rangle:\beta<\alpha\}$ is an almost disjoint family of subspaces. As $\Q_\kappa$ preserves cardinals, $|\LA_\alpha|=|\alpha|$. It remains to show that each $\LA_\alpha$ is maximal.
	
	Fix $\alpha$ as above and let $\tau$ be a nice $\Q_\kappa$-name for a subset of $E$. As $\Q_\kappa$ is ccc, there is a countable set of conditions $A\subseteq\Q_\kappa$ which decides which vectors are in $\tau$ and whether $\tau$ is an infinite-dimensional subspace. That is, if $p\forces\check{v}\in\tau$, for some $v\in E$ and $p\in\Q_\kappa$, then there is a $q\in A$ with $q\forces\check{v}\in\tau$, and likewise if $p\forces\tau\text{ is a subspace}$. $A$ is contained in
	\[
		\Q_{\kappa,S} = \{p\in\Q_\kappa:(\alpha,\gamma)\in F_p\Rightarrow\gamma\in S\}
	\]
	for some countable $S\subseteq\alpha$. Suppose that
	\[
		p\forces \text{$\tau$ is an infinite-dimensional subspace of $E$ and $\forall\gamma\in\check{S}(\dim(\tau\cap \langle \dot{G}_{\alpha,\gamma}\rangle)<\infty)$}
	\]
	for $p\in\Q_{\kappa,S}$. Fix $\xi\in\alpha\setminus S$. We claim that for all $M>0$, the set of conditions $q\in\Q_\kappa$ such that
	\[
		q\forces\exists v>M(v\in\tau\cap\langle \dot{G}_{\alpha,\xi}\rangle)
	\]
	is dense below $p$. Let $p'\leq p$. We may assume that $(\alpha,\xi)\in F_{p'}$. Let $p''=p'\restr(\{(\alpha,\gamma):\gamma\in S\}\times n_{p'})\in\Q_{\kappa,S}$. Then, $p''\leq p$, and so
	\[
		p''\forces \text{$\tau$ is an infinite-dimensional subspace of $E$ and $\forall\gamma\in\check{S}(\dim(\tau\cap \langle \dot{G}_{\alpha,\gamma}\rangle)<\infty)$}
	\]
	By Lemmas \ref{lem:extend} and \ref{lem:extend_n_p}, there is a $p'''\leq p''$ in $\Q_{\kappa,S}$ and a $v>M$ so that
	\[
		p'''\forces \check{v}\in\tau\land \forall (\alpha,\gamma)\in \check{F}_{p''}(\check{v}\notin\langle \dot{G}_{\alpha,\gamma}\rangle),
	\]
	and moreover, there is a condition $q\in\Q_\kappa$ so that $F_q=F_{p'}\cup F_{p'''}$, $n_q=n_{p'''}+1$, $q(\alpha,\xi,n_{p'''})=v$, and $q\leq p'$. But then,
	\[
		q\forces\exists v>M(v\in\tau\cap\langle \dot{G}_{\alpha,\xi}\rangle),
	\]
	as claimed. Thus, $\LA_\alpha$ is forced to be a mad family of subspaces.
	
	That $\cc\leq\kappa^{\aleph_0}$ in $\V[G]$ follows from standard facts about ccc forcing (cf.~Lemma 5.13 of Ch.~VII in \cite{MR597342}).
\end{proof}

\section{Definability and Ramsey theory}\label{sec:def}

In \cite{MR0491197}, Mathias showed that there are no analytic mad families on $\omega$. His proof proceeds by showing that, given an infinite almost disjoint family $\LA$ on $\omega$, the set $\LH$ of subsets of $\omega$ not covered by a finite union of elements of $\LA$ is a selective coideal.\footnote{This is shown for infinite mad families in Proposition 0.7 of \cite{MR0491197}, but the assumption of maximality is not necessary, see Example 2 on p.~35 of \cite{MR1442262}.} Here, a \emph{coideal} is the complement of an ideal of subsets of $\omega$, and \emph{selectivity} refers to closure under a certain kind of diagonalization. Were $\LA$ analytic, an application of the main Ramsey-theoretic dichotomy of \cite{MR0491197} shows that there must be an infinite set $x\in\LH$ none of whose infinite subsets are in the $\subseteq$-downwards closure of $\LA$. Such an $x$ witnesses that $\LA$ fails to be maximal.
%Under suitable large cardinal hypothesis, the extensions of Mathias' theorem to sets in $\L(\R)$ of certain models shows that there are no infinite mad families in $\L(\R)$.

We would like to replicate this argument to prove that there are no infinite analytic mad families of subspaces of $E$, considered as subsets of the product space $2^E$. As is the case for mad families on $\omega$, such a result would be sharp: assuming $\V=\L$, the methods in \cite{MR983001} can be adapted to show that there is a coanalytic mad family of subspaces. This na\"ive approach runs into several problems, which we discuss below.

Let's first consider the setting where $F$ is a finite field, in which case almost disjoint sub\emph{spaces} of $E$ are also almost disjoint as sub\emph{sets} of $E$. This suggests the following strategy: Suppose that $\LA$ is an infinite analytic almost disjoint family of subspaces of $E$ and let $\LH$ be the collection of all sub\emph{sets} of $E$ which are not covered by a union of finitely many elements of $\LA$. Then, $\LH$ is a selective coideal of subsets of $E$. Applying Mathias' dichotomy as above, we obtain an infinite subset $X\in\LH$ all of whose further subsets are disjoint from the downwards closure of $\LA$. If $\LA$ were maximal, then we would obtain the desired contradiction provided $X$ contains an infinite-dimensional subspace. However, there is no a priori reason why $X$ ought to contain such a subspace.

In the event that $|F|=2$, hope is provided by Hindman's theorem \cite{MR0349574}, one formulation of which says that the collection $\LB$ of all subsets of $E$ which contain an infinite-dimensional block subspace is a coideal. It would suffice, then, to show that $\LH\cap\LB$ is a selective coideal. As the union of two ideals is an ideal if and only if one contains the other, we would need to have that $\LH\subseteq\LB$ (clearly, $\LB\not\subseteq\LH$). Unfortunately, this is \emph{never} true: take $X\in\LH$ which has infinite intersection with infinitely many elements of $\LA$ and build a block sequence $Y$ in $X$ with the same property. Taken as a set, $Y\in\LH$ but $Y$ contains no subspaces. This argument can be adapted to show that the family of block sequences in $E$ whose spans are in $\LH$ fails to be a coideal in the associated Ramsey space of all block sequences, in the sense of \cite{MR3748482}.

We now turn to a strategy based on the Ramsey-theoretic results in \cite{MR3864398} for block sequences in vector spaces over an arbitrary countable field $F$.%This work refines prior work of Gowers \cite{MR1954235} and Rosendal \cite{MR2604856}.

Following \cite{MR3864398}, we let $\bb^\infty(E)$ denote the space of all infinite block sequences in $E$, which inherits a Polish topology from $E^\omega$ that is compatible with the Borel structure of $2^E$. For $X,Y\in\bb^\infty(E)$, we write $X\preceq Y$ if $\langle X\rangle\subseteq\langle Y\rangle$, and $X\preceq^*Y$ if $X/n\preceq Y$ for some $n$, where $X/n$ denotes the tail subsequence of $X$ consisting of those vectors in $X$ with supports above $n$. Note that $\langle X/n\rangle=\langle X\rangle/n$, where the latter was defined for subspaces in \S\ref{sec:card_ZFC}. A nonempty subset of $\bb^\infty(E)$ is a \emph{family} if it is  closed upwards with respect to $\preceq^*$. If $X\in\LH$, we write $\LH\restr X=\{Y\in\LH:Y\preceq X\}$. The key notions from \cite{MR3864398} are as follows:

\begin{defn}
	A family $\LH\subseteq\bb^\infty(E)$ is:
	\begin{enumerate}
		\item a \emph{$(p)$-family}, or has the \emph{$(p)$-property}, if whenever $X_0\succeq X_1\succeq \cdots$ is a decreasing sequence with each $X_n\in\LH$, there is a $Y\in\LH$ with $Y\preceq^* X_n$ for all $n\in\omega$.
		\item \emph{full} if whenever $D\subseteq E$ (not necessarily a subspace) and $X\in\LH$ are such that for every $Y\in\LH\restr X$, there is a $Z\preceq Y$ with $\langle Z\rangle\subseteq D$, then there is a $Z\in\LH\restr X$ with $\langle Z\rangle\subseteq D$.	
		\item a \emph{$(p^+)$-family} if it is full and has the $(p)$-property.
	\end{enumerate}
\end{defn}

%The main results of \cite{MR3864398} are stated in terms of the following games:

\begin{defn}
	The \emph{Gowers game} \cite{MR1954235} played below $X\in\bb^\infty(E)$, denoted $G[X]$, is as follows: Players I and II alternate, with player I going first and playing block sequences $X_k\preceq X$, and player II responding with vectors $y_k\in \langle X_k\rangle$ subject to the constraint $y_k<y_{k+1}$, for $k\in\omega$. The block sequence $(y_k)$ is the \emph{outcome} of a play of the game.
\end{defn}

	A \emph{strategy} for II in $G[X]$ is a function $\alpha$ taking sequences $(X_0,\ldots,X_k)$ of possible prior moves by I to vectors $y\in \langle X_{k}\rangle$, with $\alpha(X_0,\ldots,X_{k-1})<y$ for all $k$. Given a set $\A\subseteq\bb^\infty(E)$, we say that $\alpha$ is a strategy for \emph{playing into $\A$} if whenever II follows $\alpha$ (that is, at each turn, given as input I's prior moves, II plays the output of $\alpha$), the resulting outcome lies in $\A$. These notions are defined likewise for I.

\begin{defn}
	The \emph{infinite asymptotic game} \cite{MR2566964} \cite{MR2604856} played below $X$, denoted $F[X]$, is as follows: Players I and II alternate, with player I going first and playing $n_k\in\omega$, and player II responding with vectors $y_k\in \langle X/n_k\rangle$ subject to the constraint $y_k<y_{k+1}$, for $k\in\omega$. Again, $(y_k)$ is the \emph{outcome} of a play of the game.	
\end{defn}

\emph{Strategies} for I and II in $F[X]$ are defined as above for $G[X]$, as is the notion of having a strategy for \emph{playing into} a set. In both games, we will be agnostic about which player ``wins''.

\begin{defn}
	A family $\LH\subseteq\bb^\infty(E)$ is \emph{strategic} if whenever $\alpha$ is strategy for II in $G[X]$, when $X\in\LH$, there is an outcome of $\alpha$ which is in $\LH$.	
\end{defn}

It is proved in \cite{MR3864398} that any sufficiently generic filter for $(\bb^\infty(E),\preceq^*)$, viewed as a $\sigma$-closed notion of forcing, is a strategic $(p^+)$-family.

The following theorem from \cite{MR3864398} was originally proved by Rosendal \cite{MR2604856} in the case $\LH=\bb^\infty(E)$, which in turn was a discretized version of the dichotomy for block sequences in Banach spaces proved by Gowers in \cite{MR1954235}.

\begin{thm}[Theorem 1.1 in \cite{MR3864398}]\label{thm:local_Rosendal}
	Let $\LH\subseteq\bb^\infty(E)$ be a $(p^+)$-family. If $\A\subseteq\bb^\infty(E)$ is analytic, then there is a $X\in\LH$ such that either
	\begin{enumerate}[label=\rm{(\roman*)}]
		\item I has a strategy in $F[X]$ for playing into $\A^c$, or
		\item II has a strategy in $G[X]$ for playing into $\A$.
	\end{enumerate}
\end{thm}

Assuming large cardinal hypotheses, and that $\LH$ is strategic, Theorem \ref{thm:local_Rosendal} can be extended to all sets $\A$ in $\L(\R)$ (Theorem 1.3 in \cite{MR3864398}).

In what follows, if $\LA$ is an infinite almost disjoint family of subspaces of $E$ (notably, the elements of $\LA$ need not be block subspaces), we let
\[
	\LH(\LA)=\{X\in\bb^\infty(E):\exists^\infty Y\in\LA(\dim(\langle X\rangle\cap Y)=\infty)\}.
\]
Note that $\LH(\LA)$ is always nonempty, as it contains $(e_n)$, is closed upwards with respect to $\preceq^*$, and is thus a family. We let
\[
	\bar{\LA}=\{X\in\bb^\infty(E):\exists Y\in\LA(\langle X\rangle\subseteq Y)\}. 
\]
Note that $\bar{\LA}\cap\LH(\LA)=\emptyset$, and that if $\LA$ is analytic, so is $\bar{\LA}$.

\begin{lemma}\label{lem:strats_H(A)}
	If $\LA$ is an infinite almost disjoint family of subspaces of $E$, then for any $X\in\LH(\LA)$,
	\begin{enumerate}
		\item I and II have strategies in $G[X]$ and $F[X]$, respectively, for playing into $\LH(\LA)$.
		\item If $\LA$ is maximal, then I and II have strategies in $G[X]$ and $F[X]$, respectively, for playing into $\bar{\LA}$.
	\end{enumerate}
\end{lemma}

\begin{proof}
	(a) Fix an enumeration $(Y_n)$ of a countably infinite subset of $\LA$, each $Y_n$ having infinite-dimensional intersection with $\langle X\rangle$, in such a way that each $Y_n$ is repeated infinitely often. To see that I has a strategy in $G[X]$ for playing into $\LH(\LA)$, let I play an infinite block sequence in $\langle X\rangle\cap Y_n$ on their $n$th move. The resulting outcome will have infinitely many entries in each $Y_n$ and is thus in $\LH(\LA)$. To see that II has a strategy in $F[X]$ for playing into $\LH(\LA)$, let II play the first element of $Y_n$ they can on their $n$th move.%The resulting outcome will be compatible with each $Y_n$ and thus in $\LH(\LA)$.

	\noindent(b) Suppose that $\LA$ is maximal. Take $Y\in\LA$ having infinite-dimensional intersection with $\langle X\rangle$. To see that I has a strategy in $G[X]$ for playing into $\bar{\LA}$, let I play, repeatedly, any infinite block sequence $Z$ contained in $\langle X\rangle\cap Y$. The resulting outcome will be below $Y$. To see that II has a strategy in $F[X]$ for playing into $\bar{\LA}$, observe that so long as II plays in $Y$, which they may always do, the outcome will be contained in $Y$. 	
\end{proof}

\begin{lemma}\label{lem:cont_mod_finite}
	For $X$ an infinite-dimensional subspace, $Y$ a block subspace, and $z_0<\cdots<z_\ell$ in $E$, if $X\subseteq Y+\langle z_0,\ldots,z_\ell\rangle$, then there is an $M$ such that $X/M\subseteq Y$.	
\end{lemma}%DOES Y REALLY NEED TO BE BLOCK HERE? PROBABLY NOT...

\begin{proof}
	Let $(y_n)$ be a block basis for $Y$. Let $N = \max\{\supp(z_i):i\leq\ell\}$ and suppose that $y_0,\ldots,y_k$ are those basis vectors in $Y$ whose supports are not above $N$. Let $M=\max\{N,\max(\supp(y_k))\}$. We claim that $X/M\subseteq Y$. Take $x\in X/M$. By assumption, $x=y+w$ where $y\in Y$ and $w\in\langle z_0,\ldots,z_\ell\rangle$. Write $y=y'+y''$, where $y'\in\langle y_0,\ldots,y_k\rangle$ and $y''\in\langle y_{k+1},y_{k+2},\ldots\rangle$, so that $x-y''=y'+w$. If either side of this equation is nonzero, then $\supp(x-y'')>M$, but $\supp(y'+w)\leq M$, a contradiction. Thus, $x=y''\in Y$.
\end{proof}

\begin{lemma}\label{lem:mad_strat_p}
	If $\LA$ is an infinite mad family of subspaces, then $\LH(A)$ is strategic and has the $(p)$-property.
\end{lemma}

\begin{proof}
	That $\LH(\LA)$ is strategic is immediate from Lemma \ref{lem:strats_H(A)}(a), as whenever $\alpha$ is a strategy for II in $G[X]$, for $X\in\LH(\LA)$, we may let I use their strategy for playing into $\LH(\LA)$ in response.
	
	In what follows, if $(Z_n)$ is a sequence in $\bb^\infty(E)$ and $Z\in\bb^\infty(E)$ is such that $Z/n\preceq Z_n$ for all $n\in\omega$, we will call $Z$ a \emph{diagonalization} of $(Z_n)$. 
	
	To see that $\LH(\LA)$ has the $(p)$-property, let $X_0\succeq X_1\succeq X_2\succeq\cdots$ be a decreasing sequence contained within $\LH(\LA)$. Let $X^0\in\bb^\infty(E)$ be a diagonalization of $(X_n)$ and take $Y^0\in\LA$ having infinite-dimensional intersection with $\langle X^0\rangle$.  Following the proof of Proposition 0.7 in \cite{MR0491197}, we will construct sequences $(X^m)$ and $(Y^m)$ in $\bb^\infty(E)$ where each $Y^m$ is a distinct element of $\LA$, $\langle X^m\rangle$ has infinite-dimensional intersection with $Y^m$, and $X^m$ a further diagonalization of $(X_n)$.%and $X^m$ diagonalizes (in the sense of Definition \ref{def:diag}) $(X_n)$ with $X^m/n\preceq X_n$ for all $n$.
	
	%Given such a pair of sequences $(X^m)$ and $(Y^m)$, the proof is completed as follows: let $i:\omega\to\omega$ be an everywhere infinity-to-one surjection (i.e., for all $m\in\omega$, $i^{-1}(m)$ is infinite) and consider the  sequence of pairs $(X^{i(m)},Y^{i(m)})$. Construct $X=(x_m)$ so that each $x_m\in\langle X^{i(m)}/m\rangle\cap Y^{i(m)}$. Then, $X\in\LH(\LA)$, and moreover, for all $n$, if $x\in\langle X/n\rangle$, then $x$ is a linear combination of elements of $X^{i(m_0)}/n,\ldots,X^{i(m_k)}/n$, each of which is $\preceq X_n$. So, $X/n\preceq X_n$ for all $n$.
	
	For each $n$, construct a countably infinite pairwise disjoint family of block sequences $\LA_n$ below $X_n$ such that
	\begin{enumerate}[label=\rm{(\roman*)}]
		\item for all $Y\in\LA_n$, there is a $Y'\in\LA$ with $\langle Y\rangle\subseteq Y'$, and
		\item for all $Y\in\LA_n$, $\langle Y\rangle$ is disjoint from $Y^0$.
	\end{enumerate}
	This can be accomplished as $X_n\in\LH(\LA)$; simply take a countably infinite $\LA'_n\subseteq\LA$ not containing $Y^0$, all of whose elements have infinite-dimensional intersection with $\langle X_n\rangle$, and let $\LA_n$ be a set of block bases of subspaces witnessing this. Pairwise disjointness and disjointness from $Y^0$, for elements in $\LA_n$, can be ensured by passing to tail block sequences. Enumerate $\LA_n=\{Y^n_i:i\in\omega\}$ in such a way that each element is repeated infinitely often.
	
	Next, we build a decreasing sequence $X^0_0\succeq X^0_1\succeq X^0_2\succeq\cdots$ in $\LH(\LA)$ such that for each $n$, $X^0_n\preceq X_n$, and $\langle X^0_n\rangle$ is almost disjoint from $Y^0$. We will denote by $X^0_n=(x^0_{n,i})_{i\in\omega}$.

	Let $x^0_{0,0}$ be the first entry of $Y^0_0$. There must be a nonzero $x\in\langle Y^1_0\rangle$ above $x^0_{0,0}$ such that no linear combination of $x$ and $x^0_{0,0}$ is in $Y^0$, otherwise $Y^1_0\preceq^* Y^0$ by Lemma \ref{lem:cont_mod_finite}. Let $x^0_{1,0}=x^0_{0,1}\in Y^1_0$ be such a vector. We continue in this fashion, following the diagram in Figure \ref{fig:diag_construct}, with $\color{red}X^0_0=(x^0_{0,n})$, $\color{blue}X^0_1=(x^0_{1,n})$, $\color{olive}X^0_2=(x^0_{2,n})$, $X^0_3=(x^0_{3,n})$, etc.
	
	\begin{figure}
		\centering\scalebox{0.80}{\xymatrix@R=1.2em@C=1.5em{ 
	        &		&       &       &       &       & X_0 \ar@{-}[dl] \ar@{-}[dll] \ar@{-}[dlll] \ar@{-}[dllll] \ar@{-}[dd]\\
	 \LA_0: &\cdots & \underset{\color{red}x^0_{0,9}}{Y^0_3} & \underset{\color{red}x^0_{0,5}}{Y^0_2} & \underset{\color{red}x^0_{0,2}}{Y^0_1} & \underset{\color{red}x^0_{0,0}}{Y^0_0} &    \\
	        &		&       &       &       &       & X_1\ar@{-}[dl] \ar@{-}[dll] \ar@{-}[dlll] \ar@{-}[dllll] \ar@{-}[dd]\\
	 \LA_1: &\cdots & Y^1_3 & \underset{\color{blue}x^0_{1,5}=\color{red}x^0_{0,8}}{Y^1_2} & \underset{\color{blue}x^0_{1,2}=\color{red}x^0_{0,4}}{Y^1_1} & \underset{\color{blue}x^0_{1,0}=\color{red}x^0_{0,1}}{Y^1_0} &    \\
	        &		&       &       &       &       & X_2\ar@{-}[dl] \ar@{-}[dll] \ar@{-}[dlll] \ar@{-}[dllll] \ar@{-}[dd]\\
	 \LA_2: &\cdots & Y^2_3 & Y^2_2 & \underset{\color{olive}x^0_{2,2}=\color{blue}x^0_{1,4}=\color{red}x^0_{0,7}}{Y^2_1} & \underset{\color{olive}x^0_{2,0}=\color{blue}x^0_{1,1}=\color{red}x^0_{0,3}}{Y^2_0} &    \\
	        &		&       &       &       &       & X_3\ar@{-}[dl] \ar@{-}[dll] \ar@{-}[dlll] \ar@{-}[dllll] \ar@{-}[d]\\
	 \LA_3: &\cdots & Y^3_3 & Y^3_2 & Y^3_1 & \underset{x^0_{3,0}=\color{olive}x^0_{2,1}=\color{blue}x^0_{1,3}=\color{red}x^0_{0,6}}{Y^3_0} & \vdots  	 
}}	
	\caption{}
	\label{fig:diag_construct}
\end{figure}
	
	That is, let $x^0_{0,2}\in Y^0_1 $ be a vector above $x^0_{0,1}$ such that no linear combination of it with $x^0_{0,0}$ and $x^0_{0,1}$ lies in $Y^0$. Next, let $x^0_{0,3}=x^0_{1,1}=x^0_{2,0}\in\langle Y^2_0\rangle$ be a vector above $x^0_{0,2}$ such that no linear combination of it with $x^0_{0,0}$, $x^0_{0,1}$ and $x^0_{0,2}$ lies in $Y^0$. And so on.
	
	By construction, $X^0_0\succeq X^0_1\succeq X^0_2\succeq\cdots$ as each $X^0_n$ is a subsequence of the previous ones, and each $\langle X^0_n\rangle$ is disjoint from $Y^0$. Moreover, each $\langle X^0_n\rangle$ has infinite-dimensional intersection with $\langle Y\rangle$, for each $Y\in\LA_n$, and $X^0_n\in\LH(\LA)$. Let $X^1$ be a diagonalization of $(X^0_n)$, and thus also a diagonalization of the original $(X_n)$ as well. Let $Y^1\in\LA$ have infinite-dimensional intersection with $\langle X^1\rangle$. Note that we must have $Y^1\neq Y^0$.
	
	We continue this process to obtain $(X^m)$ and $(Y^m)$  as desired. Let $i:\omega\to\omega$ be an everywhere infinity-to-one surjection and consider the  sequence of pairs $(X^{i(m)},Y^{i(m)})$. Construct $X=(x_m)$ so that each $x_m\in\langle X^{i(m)}/m\rangle\cap Y^{i(m)}$. Then, $X\in\LH(\LA)$, and moreover, for all $n$, if $x\in\langle X/n\rangle$, then $x$ is a linear combination of elements of $X^{i(m_0)}/n,\ldots,X^{i(m_k)}/n$, each of which is $\preceq X_n$. So, $X/n\preceq X_n$ for all $n$.
\end{proof}

\begin{defn}
	An infinite mad family $\LA$ of subspaces is \emph{full} if $\LH(\LA)$ is full.	
\end{defn}

The preceding lemmas, and Theorem \ref{thm:local_Rosendal}, yield the following:

\begin{thm}\label{thm:no_full_mad}
	There are no analytic full mad families of subspaces.	%\begin{enumerate}
	%	\item If $\LA$ is analytic, then $\LA$ fails to be maximal.
	%	\item (Assume that there is a supercompact cardinal.) If $\LA$ is in $\L(\R)$, then $\LA$ fails to be maximal. 
	%\end{enumerate}
\end{thm}

\begin{proof}
	Suppose that $\LA$ was an analytic full mad family of subspaces. By Lemma \ref{lem:mad_strat_p}, $\LH$ is a $(p^+)$-family. Applying Theorem \ref{thm:local_Rosendal} to the analytic set $\bar{\LA}$, there is an $X\in\LH(\LA)$ such that either I has a strategy in $F[X]$ for playing into $\bar{\LA}^c$, or II has a strategy in $G[Y]$ for playing into $\bar{\LA}$. However, the latter contradicts Lemma \ref{lem:strats_H(A)}(a), while the former contradicts Lemma \ref{lem:strats_H(A)}(b). %Part (b) is proved similarly, using Theorem \ref{thm:local_Rosendal_L(R)}.
\end{proof}

Under large cardinal hypotheses, an identical proof, using Theorem 1.3 in \cite{MR3864398}, shows that no full mad family of subspaces can be in $\L(\R)$.

Must a mad family of subspaces be full? Unfortunately, we are only able to show that, consistently, there are such mad families. It remains an open question whether mad families must be full (we suspect not), and if not, whether full mad families exist in $\ZFC$. First, we need a variation on Theorem \ref{thm:MA_a_large}, localized to a fixed block subspace $X$. %Recall that the hypothesis of the theorem below holds under $\CH$ and $\MA(\sigma\text{-centered})$, by Proposition \ref{prop:mad_unctbl} and Theorem \ref{thm:MA_a_large}, respectively.

\begin{lemma}\label{lem:MA_local_diag}
	$(\MA_\kappa(\sigma\text{-centered}))$ Suppose that $X$ is a block subspace of $E$ and $\LC$ an almost disjoint family of block subspaces of $E$ such that each has infinite-dimensional intersection with $X$. If $|\LC|\leq\kappa$, then there is a block subspace $Y$ of $X$ almost disjoint from every element of $\LC$.
\end{lemma}

\begin{proof}
	We mimic the proof of Theorem \ref{thm:MA_a_large}. Define a poset $\P$ to be all pairs $(s,F)$ where $s$ is a finite normalized block sequence in $X$ and $F$ a finite subset of $\LC$. We order elements of $\P$ by $(s',F')\leq(s,F)$ if $s'\sqsupseteq s$, $F'\supseteq F$, and $\forall W\in F(\langle s'\rangle\cap W\subseteq\langle s\rangle)$. As before, $\P$ is $\sigma$-centered, and if $G$ is a filter in $\P$, we let $X_G=\langle\bigcup\{s:\exists F((s,F)\in G)\}\rangle$.
	
	If $W\in\LC$, then the set $D_W=\{(s,F)\in\P:W\in F\}$ is dense, and if $G\cap D_W\neq\emptyset$, then $X_G\cap W$ is finite dimensional. For $n\in\omega$, let $E_n=\{(s,F)\in\P:|s|\geq n\}$. In order to see that the sets $E_n$, as before we use Lemma \ref{lem:extend} to obtain an $M$ for which whenever $x>M$ and not in $\bigcup F$, $\langle s\concat x\rangle\cap W=\langle s\rangle\cap W$ for each of the finitely many $W\in F$. Then, for any such $x$ in $X$ (which can be found in $W\cap X$ for some $W\in\LC$), $(s\concat x,F)\leq(s,F)$.

	If $|\LC|\leq\kappa$, by $\MA_\kappa(\sigma\text{-centered})$, there is a filter $G\subseteq\P$ which meets the sets $D_W$ and $E_n$, for $W\in\LC$ and $n\in\omega$, so $Y=X_G$ is as desired.
\end{proof}

It will be useful to note that if $\LA\subseteq\LB$ are infinite almost disjoint families of subspaces, then $\LH(\LA)\subseteq\LH(\LB)$.

\begin{thm}\label{thm:full_mad}
	$(\MA(\sigma\text{-centered}))$\footnote{This result seems likely to be true under the weaker assumption that $\a_{\mathrm{vec},F}=\c$; the issue is that the relevant almost disjoint family of subspaces, namely those of the form $W\cap \langle X_\gamma\rangle$ for $W\in\LC$ in Case 2, need not be \emph{block}.} There is a full mad family of block subspaces.
\end{thm}

\begin{proof}
	We will define $\LA=\bigcup_{\alpha<\cc}\LA_\alpha$ via transfinite recursion on $\cc$. Enumerate by $\{X_\alpha:\alpha<\cc\}$ and $\{D_\alpha:\alpha<\cc\}$ all elements of $\bb^\infty(E)$ and subsets of $E$, respectively, ensuring that the enumeration $X_\alpha$ repeats each $X\in\bb^\infty(E)$ cofinally often. Fix a bijection $\langle\cdot,\cdot\rangle:\cc\times\cc\to\cc$. 
	
	Begin by letting $\LA_0$ be any countably infinite almost disjoint family of block subspaces. Given $\alpha<\cc$, suppose that for $\beta<\alpha$, $\LA_\beta$ has been defined to be an infinite almost disjoint family of block subspaces with size $\leq|\beta|+\aleph_0$, and that $\LA_\beta\subseteq\LA_\gamma$ for $\beta\leq\gamma<\alpha$. We define $\LA_\alpha$ as follows:
	
	Put $\LA_\alpha'=\bigcup_{\beta<\alpha}\LA_\beta$. If $\langle X_\alpha\rangle$ is almost disjoint from every element of $\LA_\alpha'$, then put $\LA_\alpha''=\LA_\alpha'\cup\{\langle X_\alpha\rangle\}$. If not, put $\LA_\alpha''=\LA_\alpha'$. Say $\alpha=\langle\gamma,\delta\rangle$. If $X_\gamma\notin\LH(\LA_\alpha'')$, then let $\LA_\alpha=\LA_\alpha''$. Otherwise, let $\LC$ be the collection of elements of $\LA_\alpha''$ with which $X_\gamma$ has infinite-dimensional intersection and consider the following cases:
	
	Case 1: There is a $Z\preceq X_\gamma$ such that $\langle Z\rangle$ is almost disjoint from each $Y\in\LC$ and is contained in $D_\delta$. In this case, let $\LB$ be a countably infinite almost disjoint family of infinite-dimensional subspaces below $Z$. Note that if $V\in\LB$ is compatible with some $Y\in\LA_\alpha''$, then $X_\gamma$ must be compatible with that $Y$, so $Y\in\LC$, but this yields a contradiction as $\langle Z\rangle$ must be almost disjoint from such a $Y$. Let $\LA_\alpha=\LA_\alpha''\cup\LB$, an almost disjoint family by the preceding argument. Then, $Z\in\LH(\LA_\alpha)$.
	
	Case 2: For every $Y\preceq X_\gamma$ such that $\langle Y\rangle$ is almost disjoint from every element of $\LC$, there is no $Z\preceq Y$ with $\langle Z\rangle\subseteq D_\delta$. Note that if this fails, we are in Case 1. As $|\LC|\leq|\alpha|+\aleph_0<\cc$, by $\MA(\sigma\text{-centered})$ and Lemma \ref{lem:MA_local_diag}, there is a $Y\preceq X_\gamma$ with $\langle Y\rangle$ almost disjoint from each element of $\LC$. Let $\LB$ be a countably infinite almost disjoint family below $Y$, and let $\LA_\alpha=\LA_\alpha''\cup\LB$, an almost disjoint family by the same argument as in Case 1. Then, $Y\in\LH(\LA_\alpha)$.
	
	We claim that $\LA=\bigcup_{\alpha<\cc}\LA_\alpha$ is as desired. Note that $\LH(\LA)=\bigcup_{\alpha<\cc}\LH(\LA_\alpha)$, as whenever $X\in\LH(\LA)$, a countably infinite subset of $\LA$ all compatible with $X$ must occur in some initial $\LA_\alpha$, as $\cf(\cc)>\aleph_0$. Clearly, $\LA$ is a mad family. To verify fullness, let $D\subseteq E$ and $X\in\LH(\LA)$, and suppose that for every $Y\in\LH(\LA)\restr X$, there is a $Z\preceq Y$ with $\langle Z\rangle\subseteq D$. We may take $\alpha<\cc$ large enough so that $\alpha=\langle\gamma,\delta\rangle$, $X=X_\gamma$, $D=D_\delta$, and $X_\gamma\in\LH(\LA_\alpha'')$, for $\LA_\alpha''$ as in the construction above. If Case 1 occurred for this $\alpha$, then there is a $Z\in\LH(\LA_\alpha)\restr X\subseteq\LH(\LA)\restr X$ with $\langle Z\rangle\subseteq D$. If Case 2 occurred for this $\alpha$, then there is an $Y\in\LH(\LA_\alpha)\restr X\subseteq\LH(\LA)\restr X$ having no $Z\preceq Y$ with $\langle Z\rangle\subseteq D$, contrary to assumption. Thus, there is a $Z\in\LH(\LA)\restr X$ with $\langle Z\rangle\subseteq D$, as required.
\end{proof}

The proof of Theorem \ref{thm:full_mad} can be adapted to show how to generically add a full mad family of block subspaces: Let $\P$ be the collection of all countably infinite almost disjoint families of block subspaces, ordered by reverse inclusion. It is easy to see that $\P$ is $\sigma$-closed and if $G$ is $\V$-generic for $\P$, then $\LG=\bigcup G$ is a mad family of block subspaces. The arguments in Cases 1 and 2 above show that, for $\LA\in\P$, $X\in\LH(\LA)$, and $D\subseteq E$, the set of all $\LB\in\P$ such that $\LH(\LB)$ ``witnesses fullness for $X$ and $D$'' is dense below $\LA$. In the language of \cite{MR3631288}, assuming $\MA(\sigma\text{-centered})$, full mad families of block subspaces \emph{exist generically}.

What can we say about analytic mad families of subspaces in the absence of fullness? For a family $\LH\subseteq\bb^\infty(E)$ and $X\in\LH$, the game $G_\LH[X]$ is the variant of $G[X]$ in which I is restricted to playing elements of $\LH\restr X$. A variant of Theorem \ref{thm:local_Rosendal}, Theorem 3.11.5 in \cite{Smythe_thesis}, can be used to obtain the following:

\begin{thm}
	Let $\LA$ be an infinite mad family of subspaces. If $\LA$ is analytic, then there is an $Y\in\LH(\LA)$ such that II has a strategy in $G_{\LH(\LA)}[Y]$ for playing into $\bar{\LA}$.
\end{thm}

Were $\LH(\LA)$ to be \emph{+-strategic}, that is, whenever $\alpha$ is a strategy for II in $G_{\LH(\LA)}[X]$, for some $X\in\LH(\LA)$, then there is an outcome of $\alpha$ in $\LH(\LA)$, then the conclusion of the above theorem would yield the desired contradiction. However, by Theorem 3.11.9 of \cite{Smythe_thesis}, this is equivalent to $\LH(\LA)$ being full. These observations suggests that full mad families of subspaces are analogous to \emph{+-Ramsey} mad families on $\omega$, as studied by Hru{\v{s}}{\'{a}}k in \cite{MR1802331} (see also \cite{MR3631288}).\footnote{A closer analogue to being +-Ramsey would replace player II with player I in the definition of +-strategic, however this does not seem relevant to the present situation.}% for almost disjoint families on $\omega$: A mad family $\LA$ on $\omega$ is \emph{+-Ramsey}\footnote{A closer analogue to being +-Ramsey would replace player II with player I in the definition of +-strategic, however this does not seem relevant to the present situation.} if whenever $T\subseteq\omega^{<\omega}$ is an $\LI(\LA)^+$-branching tree, there is a branch $b\in[T]$ such that $\ran(b)\in\LI(\LA)^+$. Here, given an ideal $\LI$, an $\LI^+$-branching tree $T$ is one for which each successor set $\{n:s\concat(n)\in T\}\in\LI^+$, $\LI(\LA)$ is the ideal generated by $\LA$, and $\LI(\LA)^+=\LP(\omega)\setminus\LI(\LA)$ the corresponding coideal. 

\section{Further remarks, conjectures and open questions}\label{sec:fin}

Many of the arguments above, particularly those dependent on Lemma \ref{lem:extend} or results from \cite{MR3685044}, depend on the subspaces involved being block subspaces. For this reason, we incorporated ``block'' into our definition of the cardinal $\a_{\mathrm{vec},F}$. It remains unclear whether this is necessary for our results.

\begin{ques}
	Given an infinite mad family of (arbitrary) infinite-dimensional subspaces, is there one of the same size consisting only of block subspaces? In particular, can we remove ``block'' from the definition of $\a_{\mathrm{vec},F}$?
\end{ques}

We have seen in Corollary \ref{cor:a<a_vec} that it is consistent that $\a<\a_{\mathrm{vec},F}$. As in \cite{MR3685044}, we also ask about the reverse inequality:

\begin{ques}
	Is $\a_{\mathrm{vec},F}<\a$ consistent with $\ZFC$?
\end{ques}

Given the results in \S\ref{sec:card_cons}, it would be interesting to further determine in which ``canonical models'' $\a_{\mathrm{vec},F}=\aleph_1$. In particular, as both $\a$ and $\mathrm{non}(\LM)$ are $\aleph_1$ in the Miller model (see \S11.9 in \cite{MR2768685}), we suspect that $\a_{\mathrm{vec},F}$ is as well. The paramterized $\diamondsuit$ principles of Moore, Hru\v{s}\'{a}k, and D\v{z}amonja \cite{MR2048518} provide a convenient way of isolating such results. For instance, it is shown in \cite{MR2048518} that $\diamondsuit(\b)$, which holds in the Cohen, Sacks, and random models, implies that $\a=\aleph_1$. By Corollary \ref{cor:a<a_vec}, this is not the case for mad families of block subspaces. We suspect instead that the ``correct'' $\diamondsuit$ principle for $\a_{\mathrm{vec},F}$ is $\diamondsuit(\omega^\omega,=^\infty)$ (cf.~Theorem 7.5 in \cite{MR2048518}):

\begin{conj}
	$\diamondsuit(\omega^\omega,=^\infty)$ implies that $\a_{\mathrm{vec},F}=\aleph_1$.
\end{conj}

As $\diamondsuit(\omega^\omega,=^\infty)$ holds in the Cohen and Sacks models, this would subsume Theorems \ref{thm:cohen} and \ref{thm:sacks}. Moreover, $\diamondsuit(\omega^\omega,=^\infty)$ implies that $\mathrm{non}(\LM)=\aleph_1$ and thus fails in the random model, consistent with Corollary \ref{cor:a<a_vec}.

None of the original results in this article have any dependence on $F$. What differences, if any, can arise from different choices of $F$? In particular:

\begin{ques}
	Is it consistent with $\ZFC$ that for some choice of fields $F$ and $K$ (e.g., $|F|=2$ and $K=\Q$) $\a_{\mathrm{vec},F}\neq\a_{\mathrm{vec},K}$?
\end{ques}

The main motivating question for \S\ref{sec:def} remains open:

\begin{ques}
	Does there exist an analytic mad family of subspaces of $E$?
\end{ques}

Since posing this question in an earlier version of this paper, it has been answered negatively by Horowitz and Shelah \cite{Horowitz_Shelah_no_mad_vec} in the special case when $|F|=2$. The work in \S\ref{sec:def} also raises the following:

\begin{ques}
	Must every mad family of subspaces be full? If not, does there exist (in $\ZFC$) a full mad family of subspaces?	
\end{ques}

This may be analogous to the existence (in $\ZFC$) of a +-Ramsey mad family on $\omega$, recently announced by Osvaldo Guzm\'{a}n-Gonz\'{a}lez \cite{guzman_ramsey_mad}.

\bibliography{/Users/iian/Dropbox/Mathematics/math_bib}{}
\bibliographystyle{abbrv}

\end{document}